\documentclass[a4paper]{amsart}
\usepackage[leqno]{amsmath}
\usepackage{amssymb}
\usepackage{amscd}
\usepackage{amsthm}
\usepackage{mathrsfs}
\usepackage{mathtools}
\usepackage{bbm}
\usepackage{color}
\usepackage{accents}
\usepackage{enumerate}
\usepackage{cite}
\usepackage[utf8]{inputenc}
\usepackage[all,cmtip]{xy}
\usepackage{etoolbox}
\usepackage{tikz}
\usepackage{extarrows}

\numberwithin{equation}{section}

\newcommand{\Z}{\ensuremath{\mathbb{Z}}}
\newcommand{\Q}{\ensuremath{\mathbb{Q}}}
\newcommand{\R}{\ensuremath{\mathbb{R}}}
\newcommand{\C}{\ensuremath{\mathbb{C}}}

\newcommand{\A}{\ensuremath{\mathbb{A}}}
\newcommand{\T}{\ensuremath{\mathbb{T}}}

\DeclareMathOperator{\Hom}{Hom}
\DeclareMathOperator{\End}{End}
\DeclareMathOperator{\Ext}{Ext}

\DeclareMathOperator{\Gal}{Gal}
\DeclareMathOperator{\Mod}{-Mod}

\DeclareMathOperator{\Ker}{ker}
\DeclareMathOperator{\Coker}{coker}

\DeclareMathOperator{\rk}{rk}

\DeclareMathOperator{\ord}{ord}

\DeclareMathOperator{\St}{St}
\DeclareMathOperator{\cind}{c-ind}
\DeclareMathOperator{\Coind}{Coind}
\DeclareMathOperator{\sm}{sm}

\DeclareMathOperator{\HH}{H}

\newcommand{\CG}{\hspace{1em}~^C G}
\newcommand{\cf}{{\mathbbm 1}}

\newcommand{\II}{\ensuremath{\mathbb{I}}}	
\newcommand{\VV}{\ensuremath{\mathbb{V}}}	
\newcommand{\WW}{\ensuremath{\mathbb{W}}}			
\newcommand{\pinfty}{\ensuremath{^{\p,\infty}}}

\newcommand{\Ah}{\ensuremath{\mathcal{A}}}
\DeclareMathOperator{\cont}{ct}
\DeclareMathOperator{\an}{an}
\DeclareMathOperator{\rig}{rig}
\DeclareMathOperator{\loc}{loc}
\newcommand{\m}{\ensuremath{\mathfrak{m}}}
\newcommand{\p}{\ensuremath{\mathfrak{p}}}
\newcommand{\q}{\ensuremath{\mathfrak{q}}}
\newcommand{\LI}{\mathcal{L}}
\DeclareMathOperator{\univ}{un}
\DeclareMathOperator{\pr}{pr}
\newcommand{\Hec}{\mathscr{H}}
\DeclareMathOperator{\ev}{ev}
\DeclareMathOperator{\ob}{ob}

\newcommand{\into}{\hookrightarrow}
\newcommand{\onto}{\twoheadrightarrow}
\newcommand{\too}{\longrightarrow}								
\newcommand{\mapstoo}{\longmapsto}
\newcommand{\intoo}{\lhook\joinrel\longrightarrow}
\newcommand{\tooin}{\longleftarrow\joinrel\rhook}

\newtheorem{Lem}{Lemma}[section]
\makeatletter
\newlength{\@thlabel@width}%
\newcommand{\thmenumhspace}{\settowidth{\@thlabel@width}{\itshape1.}\sbox{\@labels}{\unhbox\@labels\hspace{\dimexpr-\leftmargin+\labelsep+\@thlabel@width-\itemindent}}}
\makeatother
\newtheorem{Pro}[Lem]{Proposition}
\newtheorem{Thm}[Lem]{Theorem}
\newtheorem{Cor}[Lem]{Corollary}
\newtheorem{Con}{Conjecture}

\newtheorem{Ass}{Assumption}

\newtheorem*{Hyp}{Hypothesis}

\theoremstyle{definition}
\newtheorem{Def}[Lem]{Definition}
\newtheorem{Rem}[Lem]{Remark}
\newtheorem{Exa}[Lem]{Example}

\author[L.~Gehrmann]{Lennart Gehrmann}
\address{L.~Gehrmann \\ Fakult\"at f\"ur Mathematik \\ Universit\"at Duisburg-Essen \\ Thea-Leymann-Stra\ss e 9 \\ 45127 Essen \\ Germany}
\email{lennart.gehrmann@uni-due.de}

\title[Automorphic $\LI$-invariants for reductive groups]{Automorphic L-invariants for reductive groups}
\subjclass[2010]{Primary 11F55; Secondary 11F70, 11F75, 11F85}

\begin{document}

\begin{abstract}
Let $G$ be a reductive group over a number field $F$, which is split at a finite place $\p$ of $F$, and let $\pi$ be a cuspidal automorphic representation of $G$, which is cohomological with respect to the trivial coefficient system and Steinberg at $\p$.
We use the cohomology of $\p$-arithmetic subgroups of $G$ to attach automorphic $\LI$-invariants to $\pi$.
This generalizes a construction of Darmon (respectively Spie\ss), who considered the case $G=GL_2$ over the rationals (respectively over a totally real number field).
These $\LI$-invariants depend a priori on a choice of degree of cohomology, in which the representation $\pi$ occurs.
We show that they are independent of this choice provided that the $\pi$-isotypic part of cohomology is cyclic over Venkatesh's derived Hecke algebra.
Further, we show that automorphic $\LI$-invariants can be detected by completed cohomology.
Combined with a local-global compatibility result of Ding it follows that for certain representations of definite unitary groups the automorphic $\LI$-invariants are equal to the Fontaine--Mazur $\LI$-invariants of the associated Galois representation.
\end{abstract}

\maketitle

\tableofcontents

%%%%%%%%%%%%%%%%%%%%%%%%%%
%Introduction
%%%%%%%%%%%%%%%%%%%%%%%%%%

\section*{Introduction}
Let $f$ be a normalized newform of weight $2k$ and level $\Gamma_0(M)$.
Suppose that $M=pN$ with $p$ prime, $p\nmid N$ and that the $p$-th Fourier coefficient of $f$ is equal to $p^{k-1}$.
As a special case of the general interpolation formula one sees that the central critical value of the $p$-adic $L$-function of $f$ vanishes independently of the value of its complex counterpart.
In \cite{MTT} Mazur, Tate and Teitelbaum conjectured the existence of a constant $\LI(f)\in \C_p$ - the $\LI$-invariant of $f$ - which depends only on the restriction of the Galois representation attached to $f$ to a decomposition group at $p$, such that
\begin{align}\label{derivative}
\frac{d}{ds} \left.L_p(f,s)\right|_{s=k}= \LI(f)\cdot L(f,k).\tag{$\ast$}
\end{align}
In the special situation that $f$ corresponds to a rational elliptic curve $E$, i.e.~$k=1$ and $\Q_f=\Q$, the condition $a_p=1$ is equivalent to $E$ having split multiplicative reduction at $p$.
Thus, by Tate's $p$-adic uniformization theorem there exists a $p$-adic number $q\in \Q_p^{\times}$ of absolute value less than $1$
and an isomorphism
$$\mathbb{G}_m^{\rig}/q^{\Z}=E^{\rig}$$
of rigid analytic groups.
In that case Mazur, Tate and Teitelbaum propose the following candidate for the $\LI$-invariant of $f$:
$$\LI(f)=\frac{\log_p(q)}{\ord_p(q)}.$$

In the higher weight case several constructions of the $\LI$-invariant were proposed.
These a priori different $\LI$-invariants are known to be equal and fulfil equation \eqref{derivative} by the work of several authors (see \cite{BDI} or \cite{Co} for a more detailed discussion).
The following is an incomplete list of various constructions of $\LI$-invariants:
\begin{itemize}
\item Fontaine and Mazur (see \cite{Mazur}) define $\LI$-invariants in terms of the filtered Frobenius module associated to the local Galois representation attached to $f$.
\item In \cite{Teitelbaum} Teitelbaum defines $\LI$-invariants via $p$-adic integration theory in case that $f$ has a Jacquet--Langlands lift to a Shimura curve, which admits a Cerednik--Drinfeld uniformization.
It uses the description of the relevant space of automorphic forms as harmonic cochains on the Bruhat--Tits tree, which are invariant under a $p$-arithmetic subgroup of the group of units of a definite quaternion algebra, and Coleman's $p$-adic integration theory.
\item In analogy with the construction of Teitelbaum a candidate for $\LI(f)$ is defined by Darmon (cf.~\cite{D}) in the weight $2$ case and Orton (cf.~\cite{Orton}) for general weights in terms of harmonic cochains, which are invariant under $p$-arithmetic subgroups of $GL_2(\Q).$
\item Breuil (see \cite{Br2}) gives a definition of $\LI(f)$ by studying the $f$-isotypic component of completed cohomology of modular curves, which leads to the first instances of the $p$-adic Langlands program.
\end{itemize}

In recent years several authors generalized some of these constructions to higher rank groups.
In \cite{Sp} Spie\ss~generalizes Darmon's approach to Hilbert modular forms of parallel weight $2$ and proves the analogue of Mazur, Tate and Teitelbaum's exceptional zero conjecture in this setting.
Besser and de Shalit (see \cite{BdS}) give a generalization of both the $\LI$-invariant of Fontaine--Mazur and the one of Teitelbaum for varieties, which are $p$-adically uniformizable by Drinfeld's $d$-dimensional upper half space.
They replace Coleman's $p$-adic integration theory by Besser's theory of finite polynomial cohomology.
Finally, Ding generalizes Breuil's approach to automorphic forms on definite unitary groups, which are split at $p$ (cf.~\cite{Ding}).
He defines what he calls Breuil's simple $\LI$-invariants and shows that they are equal to Fontaine--Mazur $\LI$-invariants of two-dimensional subquotients of the associated local Galois representation, or rather of the associated $(\varphi,\Gamma)$-module.

The main aim of this article, which is carried out in Section \ref{mainsection}, is to generalize the approach of Darmon, or rather its representation-theoretic reformulation by Spie\ss, to suitable automorphic representations of higher rank reductive groups.
Let us give a rough sketch of the construction, in which we ignore all kinds of class number issues.
Let $\pi$ be a cuspidal automorphic representation of a semi-simple reductive group $G$ over a number field $F$ such that 
\begin{enumerate}[(i)]
\item\label{first} $\pi$ is cohomological with respect to the trivial coefficient system,
\item\label{second} there is a finite place $\p$ of $F$ such that $G$ is split at $\p$ and the local component $\pi_\p$ is the (complex, smooth) Steinberg representation $\St^{\infty}_{G_\p}(\C)$ of $G_\p=G(F_\p)$ and
\item\label{third} a form of strong multiplicity one holds for $\pi$. 
\end{enumerate}
(See Section \ref{Setup} for a complete list of assumptions we impose on $\pi$ and $G$.)
To simplify notation we assume in this introduction that the finite part of $\pi$ can be defined over the rationals.

By property \eqref{first} we can find an arithmetic subgroup $\Gamma \subset G(F)$ such that the $\pi$-isotypic part of $\HH^{\ast}(\Gamma,\Q)$ is non-zero.
Evaluation at an Iwahori-fixed vector yields a map
$$\HH^{\ast}(\Gamma^{\p},\Hom(\St^{\infty}_{G_\p}(\Q),\Q))\too \HH^{\ast}(\Gamma,\Q)$$
for a suitable $\p$-arithmetic subgroup $\Gamma^{\p}\subset G(F)$.
By analyzing the resolution of the Steinberg representation coming from its interpretation as the cohomology with compact supports of the Bruhat--Tits building of $G_\p$ one sees that the map induces an isomorphism on $\pi$-isotypic components (see Proposition \ref{dimensions}).

We fix a Borel subgroup of the split group $G_{F_\p}$.
Let $\Delta$ be the corresponding set of simple roots.
For a subset $J\subseteq\Delta$ we denote by $v^{\infty}_{J}(\Q)$ the associated smooth generalized Steinberg representation of $G_\p$, e.g.~$v^{\infty}_{\emptyset}(\Q)=\St^{\infty}_{G_\p}(\Q)$ and $v^{\infty}_\Delta(\Q)=\Q$.
By a theorem of Dat and Orlik (see \cite{Dat} and \cite{Orlik}) the space of smooth $|I|$-extensions $\Ext_{\infty}^{|I|}(v^{\infty}_I(\Q),\St^{\infty}_{G_\p}(\Q))$ is one-dimensional.
A choice of generator of said Ext-group yields a map
$$\HH^{\ast}(\Gamma^{\p},\Hom(\St^{\infty}_{G_\p}(\Q),\Q))\too \HH^{\ast+|I|}(\Gamma^{\p},\Hom(v^{\infty}_I(\Q),\Q)).$$
A crucial result of the article is that this map induces an isomorphism on $\pi$-typical components (see Corollary \ref{smoothcup}).
In the $PGL_2$-case this result was previously proven by considering an explicit resolution of said extensions by compactly induced representations (cf.~Lemma 6.2 of \cite{Sp}).
We use a different, less explicit method:
by a spectral sequence argument we reduce the proof to the vanishing of certain $\Ext$-groups, which is an easy corollary of the results of Dat and Orlik (see Corollary \ref{extvanishing}).

By a theorem of Borel and Serre the smooth Steinberg representation $\St^{\infty}_{G_\p}(\Q)$ has an integral model, which has a resolution by representations which are compactly induced from finitely generated $\Z$-modules (see Theorem \ref{resolution1}).
This implies that the natural map
\begin{align}\label{intro}\HH^{\ast}(\Gamma^{\p},\Hom(\St^{\infty}_{G_\p}(\Z),\Z))\otimes \Q\too\HH^{\ast}(\Gamma^{\p},\Hom(\St^{\infty}_{G_\p}(\Q),\Q)) \tag{$\dagger$} \end{align}
is an isomorphism (see Proposition \ref{FlachundNoethersch}).
This allows us to take cup products with more general extension classes, which we describe in the following:
For a subset $J\subseteq\Delta$ we denote by $\VV^{\an}_J(\Q_p)$ the associated locally analytic generalized Steinberg representation.
In Section \ref{Extensions2} we construct a natural isomorphism
$$\Hom_{\cont}(F_\p^{\times},\Q_p)\too\Ext^{1}_{\an}(v^{\infty}_I(\Q_p),\VV^{\an}_J(\Q_p)),\lambda \mapstoo \mathcal{E}^{\an}_{I,J}(\lambda\circ i)$$
for any subset $J\subseteq I\subseteq \Delta$ with $|I|=|J|+1$.
Here $i\in I$ denotes the unique simple root not contained in $J$.
This is a slight generalization of a result of Ding, who proved the corresponding result for $G=GL_n$ and $I=\emptyset$ (cf.~\cite{Ding}, Section 2.2).

Let $\St^{\an}_{G_\p}(\Q_p)$ be the locally analytic Steinberg representation and $i\in \Delta$ a simple root.
Via the isomorphism \eqref{intro} and basic theory of $p$-adic integration we can define a cup product pairing
$$\HH^{d}(\Gamma^{\p},\Hom(\St^{\infty}_{G_\p}(\Q),\Q_p))\times \Hom_{\cont}(F_\p^{\times},\Q_p)\xrightarrow{\cup } \HH^{d+1}(\Gamma^{\p},\Hom(v^{\infty}_{\left\{i\right\}}(\Q),\Q_p)).$$
Thus, by restricting to the $\pi$-isotypic component and to an isotypic component of the action of $\pi_0(G_\infty)$ we get a map
$$c^{(d)}_{i}(\lambda)[\pi]^{\epsilon}\colon \HH^{d}(\Gamma^{\p},\Hom(\St^{\infty}_{G_\p}(\Q),\Q_p))^{\epsilon}[\pi] \to \HH^{d+1}(\Gamma^{\p},\Hom(v^{\infty}_{\left\{i\right\}}(\Q),\Q_p))^{\epsilon}[\pi]$$
for any $\lambda \in \Hom_{\cont}(F_\p^{\times},\Q_p)$.
Let $q$ be the lowest cohomological degree, in which $\pi$ occurs.
We define $\LI_{i}^{(d)}(\pi,\p)^{\epsilon}\subseteq \Hom_{\cont}(F_\p^{\times},\Q_p)$ as the kernel of the map $\lambda \mapsto c^{(d+q)}_{i}(\lambda)[\pi]^{\epsilon}$.
By Corollary \ref{smoothcup} the $\LI$-invariant $\LI_{i}^{(d)}(\pi,\p)^{\epsilon}$ does not contain the subspace of smooth homomorphisms and, therefore, its codimension is at least one.
If the multiplicity of $\pi$ is equal to one, we conclude that the codimension of $\LI_{i}^{(0)}(\pi,\p)^{\epsilon}$ is equal to one.

We end Section \ref{mainsection} by giving some easy properties of these $\LI$-invariants.
For example, in Section \ref{properties} we study their behaviour under twisting by characters and Galois actions.
We show that one could define automorphic $\LI$-invariants also in terms of cohomology with compact supports in Section \ref{Compact}. 

In Section \ref{DHa} we give a generalization of the main result of \cite{Ge3}.
More precisely, we show that automorphic $\LI$-invariants are independent of the cohomological degree $d$ provided that the $\pi$-isotypic component of cohomology is cyclic over the derived Hecke algebra introduced by Venkatesh in \cite{Ve}.

Conjecturally one can associate a $p$-adic Galois representation $\rho_\pi$ to $\pi$, which takes values in the so-called $C$-group of $G$.
The fact that $\pi$ is Steinberg at $\p$ should imply that the restriction $\rho_{\pi,\p}$ of $\rho_\pi$ to the local Galois group at $\p$ is totally reducible.
In Section \ref{Galois} we define for each simple root $i\in \Delta$ the $\LI$-invariant $\LI_i(\rho_{\pi,\p})$ as the Fontaine--Mazur $\LI$-invariants of the two-dimensional subquotient of $\rho_{\pi,\p}$ cut out by $i$.
The rest of Section \ref{someconjectures} is devoted to showing that automorphic and Galois theoretic $\LI$-invariants agree in a very special situation that was considered before by Ding.
In particular, we assume that $G$ is a definite unitary group and $\pi$ is spherical except at $\p$ and another place (see Section \ref{Galoisproof} for further assumptions we have to impose on $G$ and $\pi$).

In order to do this, we first show that that the $0$-th degree $\LI$-invariant $\LI_{i}^{(0)}(\pi,\p)^{\epsilon}$ is detected by completed cohomology (see Proposition \ref{last} and \ref{lastlast}).
This is a generalization of a theorem of Breuil, who considered the case of modular curves (see \cite{Br2}, Theorem 1.1.5).
Then, in the above mentioned special situation one can use the local-global compatibility theorem of Ding (see \cite{Ding}, Theorem 1.2) to obtain the equality of $\LI$-invariants.

\subsection*{Notations.}
All rings are assumed to be commutative and unital.
The group of invertible elements of a ring $R$ will be denoted by $R^{\times}$.
If $R$ is a ring and $G$ a group, we will denote the group algebra of $G$ over $R$ by $R[G]$.
Given topological groups $H$ and $G$ we write  $\Hom_{\cont}(H,G)$ for the space of continuous homomorphism from $H$ to $G$.
Let $\chi\colon G\to R^{\times}$ be a character.
We write $R[\chi]$ for the $G$-representation, which underlying $R$-module is $R$ itself and on which $G$ acts via the character $\chi$.
The trivial character will be denoted by $\cf$.

\subsection*{Acknowledgements.}
It is my pleasure to thank Vytautas Paškūnas for extensive discussions about this manuscript, mathematics and life in general.
I am grateful to Yiwen Ding for inviting me for a stay at Beijing International Center for Mathematical Research, during which part of this article was written.
I thank Michael Spie\ss~for helpful conversations on completed cohomology and the anonymous referee for his detailed list of comments that vastly improved the exposition of this paper.
Finally, I would like to thank all past and present members of ESAGA, who taught me new and exciting mathematics during the past five years.

%%%%%%%%%%%%%%%%%
%Setup
%%%%%%%%%%%%%%%%%

\section{The setup}\label{Setup}
We fix an algebraic number field $F$ with ring of integers $\mathcal{O}$.
In addition, we fix a finite place $\p$ of $F$ lying above the rational prime $p$ and choose embeddings
$$\overline{\Q_\p} \tooin \overline{\Q}\intoo \C.$$

If $v$ is a place of $F$, we denote by $F_{v}$ the completion of $F$ at $v$.
If $v$ is a finite place, we let $\mathcal{O}_{v}$ denote the valuation ring of $F_{v}$ and $\ord_{v}$ the additive valuation such that $\ord_{v}(\varpi)=1$ for any local uniformizer $\varpi\in\mathcal{O}_{v}$.
We write $\mathcal{N}(v)$ for the cardinality of the residue field of $\mathcal{O}_{v}$.

Let $\A$ be the adele ring of $F$, i.e~the restricted product over all completions $F_{v}$ of $F$.
We write $\A^\infty$ (respectively $\A\pinfty$) for the restricted product over all completions of $F$ at finite places (respectively finite places different from $\p$).

If $H$ is an algebraic group over $F$ and $v$ is a place of $F$, we write $H_v=H(F_v)$.
We put $H_\infty=\prod_{v\mid\infty}H_v$.

Throughout the article we fix a connected, semi-simple algebraic group $G$ over $F$ of $F$-rank $l$.
We assume that the base change $G_{F_\p}$ of $G$ to $F_\p$ is split.
Moreover, we assume that $p$ is odd, if the root system of $G$ has irreducible components of type $B$, $C$ or $F_4$, and, if it has irreducible components of type $G_2$, we assume that $p >3$.\footnote{These are exactly the same restrictions as in \cite{OrSchraen} .}
Let $K_\infty\subseteq G_\infty$ denote a fixed maximal compact subgroup.
The integers $\delta$ and $q$ are defined via
\begin{align*}\delta&=\rk G_\infty - \rk K_\infty \\
\intertext{and}
2 q+\delta&=\dim G_\infty - \dim K_\infty.
\end{align*}
That the a priori rational number $q$ is integral follows from the equality
$$\dim H \equiv \rk H \bmod 2,$$
which holds for every reductive group $H$ by the root space decomposition. 

At last, we fix a cuspidal automorphic representation $\pi=\otimes_v \pi_v$ of $G(\A)$ with the following properties:
\begin{itemize}
\item $\pi$ is cohomological with respect to the trivial coefficient system,
\item $\pi$ is tempered at $\infty$ and
\item $\pi_\p$ is the (smooth) Steinberg representation $\St^{\infty}_{G_{\p}}(\C)$ of $G_\p$.
\end{itemize}

\begin{Hyp}[SMO]\label{Hyp}
We assume that the following strong multiplicity one hypothesis on $\pi$ holds:
If $\pi^{\prime}$ is an automorphic representation of $G$ such that 
\begin{itemize}
\item $\pi_v^{\prime}\cong\pi_v$ for all finite places $v\neq\p$,
\item $\pi_\p^{\prime}$ has a Parahori-invariant vector and
\item $\pi_\infty$ has non-vanishing $(\mathfrak{g},K_{\infty}^{\circ})$-cohomology,
\end{itemize}
then $\pi$ is cuspidal, $\pi^{\prime}_\p=\pi_\p$ and $\pi_\infty$ is tempered.
\end{Hyp}
We define $m_\pi$ to be the sum of the multiplicities in the space of automorphic forms of all (distinct) $\pi^\prime$ as above.

\begin{Rem}
One could weaken the assumptions on $G$.
For example, it is enough to assume that $G$ is reductive and not necessarily semi-simple.
To ease notations, e.g.~one does not have to fix a central character, we stick to the semi-simple case.
The condition that the group is split at $\p$ could be removed by slightly extending known results from the literature.
Anyway, it turns out that the split situation is the richest one.
For example, in case $G_\p$ is compact our construction would be empty (see Remark \ref{split} for more details).

Of course, one would like to extend the construction to representations which are cohomological with respect to an arbitrary coefficient system.
Most of the results carry over to the more general situation. 
But the existence of nice integral structures on algebraic twists of Steinberg representations is not known.
This seems to be a difficult problem (see Remark \ref{localg}).

The strong multiplicity one hypothesis is known in some examples, most prominently it is known for all cuspidal automorphic representations in case $G=PGL_n$. Moreover, in that case we have $m_\pi=1$.
\end{Rem}

%%%%%%%%%%%%%%%%%%
%Local considerations
%%%%%%%%%%%%%%%%%%

\section{Local considerations}
In this section we recollect results about generalized Steinberg representations of the group $G_\p$.
If $H$ is an algebraic group over $F_\p$, e.g.~a subgroup of the base change $G_{F_\p}$ of $G$ to $F_\p$, we also denote the group of $F_\p$-valued points of $H$ by $H$.

Given a locally profinite group $H$ and a ring $R$ we denote by $\mathfrak{C}^{\sm}_{R}(H)$ the category of smooth $R[H]$-modules.

\subsection{Resolutions of smooth representations}\label{Resolutions}
We introduce the class of flawless smooth representations.
In a previous article  of the author (cf.~\cite{Ge}) these representations were called homologically of finite type.
The cohomology of such representations (or rather their algebraic duals) is particularly well-behaved (see Proposition \ref{FlachundNoethersch}).

Let $H$ be a connected, semi-simple algebraic group over $F_\p$, $K\subseteq H$ a compact, open subgroup and $L$ an object in $\mathfrak{C}^{\sm}_{R}(K).$
The \emph{(smooth) compact induction} $\cind^{H}_{K}L$ of $L$ from $K$ to $H$ is the space of all functions $f\colon H\to L$ that satisfy
\begin{itemize}
\item $f(hk)=k^{-1}f(h)$ for all $k\in k, h\in H$ and
\item $f$ has compact support.
\end{itemize}
Compact induction $\cind^{H}_{K}L$ defines a smooth $H$-representation via left translation.

\begin{Def}
Let $R$ be a ring and $H$ a connected, semi-simple algebraic group over $F_\p$.
An object $M\in\ob(\mathfrak{C}^{\sm}_{R}(H))$ is called \emph{flawless} if 
\begin{itemize}
\item $M$ is projective as an $R$-module and
\item there exists a finite length resolution
$$0\too P_m\too\cdots\too P_0\too M \too 0$$
in $\mathfrak{C}^{\sm}_{R}(H)$, where each $P_i$ is a finite direct sum of modules of the form
$$\cind_{K}^{H} L$$
with $K\subseteq H$ a compact, open subgroup and $L\in\ob(\mathfrak{C}^{\sm}_{R}(K))$ finitely generated projective over $R$.
\end{itemize}
\end{Def}

\begin{Rem}
Suppose $R=\Omega$ is a field of characteristic $0$.
Then the representations $P_i$ in the definition above are finitely generated, projective objects in $\mathfrak{C}^{\sm}_{\Omega}(H)$.
\end{Rem}

In the classical smooth representation theory of $p$-adic groups the property of being flawless is ubiquitous as the following theorem of Schneider and Stuhler (cf.~\cite{SSred}) shows.
\begin{Thm}[Schneider--Stuhler]\label{admissible}
Let $H$ be a connected, semi-simple algebraic group over $F_\p$ and $\Omega$ a field of characteristic $0$.
Every admissible representation $V\in\ob(\mathfrak{C}^{\sm}_{\Omega}(H))$ of finite length is flawless.
\end{Thm}

\begin{Lem}\label{flawlessproperties}
Let $R$ be a ring and $H_1, H_2$ connected, semi-simple algebraic groups over $F_\p$.
\begin{enumerate}
\item\label{flawlessbasechange} If $M\in\ob(\mathfrak{C}^{\sm}_{R}(H_1))$ is flawless and $S$ is an $R$-algebra, then
$$M\otimes_R S \in\ob(\mathfrak{C}^{\sm}_{S}(H_1))$$
is flawless.
\item\label{flawlesstensor} Let $M_1\in \ob(\mathfrak{C}^{\sm}_{R}(H_1))$ and $M_2\in \ob(\mathfrak{C}^{\sm}_{R}(H_2))$ be flawless.
Then the tensor product
$$M_1\otimes_R M_2\in \ob(\mathfrak{C}^{\sm}_{R}(H_1\times H_2))$$
is flawless.
\item\label{flawlesspullback}
If $\phi\colon H_1\to H_2$ is an isogeny of algebraic groups and $M\in \ob(\mathfrak{C}^{\sm}_{R}(H_2))$ a flawless representation of $H_2$, then
$$\phi^\ast(M)\in \ob(\mathfrak{C}^{\sm}_{R}(H_1))$$
is flawless.
\end{enumerate}
\end{Lem}
\begin{proof}
The first claim follows from the fact that for any open, compact subgroup $K_1\subseteq H_1$ and any $L_1\in\ob(\mathfrak{C}^{\sm}_{R}(K_1))$ the canonical map
$$(\cind_{K_1}^{H_1} L_1)\otimes_R S \too \cind_{K_1}^{H_1} (L_1\otimes_R S)$$
is an isomorphism.

Let $K_i\subseteq H_i$ be compact, open subgroups and $L_i\in\ob(\mathfrak{C}^{\sm}_{R}(K_i))$ smooth representations.
The $(H_1\times H_2)$-equivariant pairing
\begin{align*}
(\cind_{K_1}^{H_1} L_1)\times (\cind_{K_2}^{H_2} L_2)&\too \cind_{K_1\times K_2}^{H_1\times H_2} L_1\otimes_R L_2\\
(f_1,f_2)&\mapstoo \left[(g_1,g_2)\mapstoo f_1(g_1)\otimes f_2(g_2)\right]
\end{align*}
induces an isomorphism
$$
(\cind_{K_1}^{H_1} L_1)\otimes_R (\cind_{K_2}^{H_2} L_2)\xlongrightarrow{\cong}\cind_{K_1\times K_2}^{H_1\times H_2} (L_1\otimes_R L_2).
$$
Thus, given flawless resolutions $P_\bullet^{[i]}$ of objects $M_i\in \ob(\mathfrak{C}^{\sm}_{R}(H_i))$, then the tensor product $P_\bullet^{[1]}\otimes_R P_\bullet^{[2]}$ of chain complexes is a flawless resolution of $M_1\otimes_R M_2.$

For the last assertion we first note that any isogeny $\phi\colon H_1\to H_2$ of semi-simple groups has to be central and that the corresponding group homomorphism on $F_\p$-valued points has finite kernel and cokernel.
Let $K_2\subseteq H_2$ be an open, compact subgroup.
We fix a set of representatives $g_1,\ldots,g_n\in H_2$ of the double coset $\phi(H_1)\backslash H_2/K_2$
and put $K_1^{g_i}=\phi^{-1}(g_i K_2 g_i^{-1}).$
Given a representation $L_2\in\ob(\mathfrak{C}^{\sm}_{R}(K_2))$ Mackey's decomposition formula tells us that 
$$\phi^\ast(\cind_{K_2}^{H_2}L_2)= \bigoplus_{i=1}^{n} \cind_{K_1^{g_i}}^{H_1}L_2,$$
where $K_1^{g_i}$ acts on $L_2$ via $k.l=(g_i^{-1}\phi(k)g_i).l$ for all $k\in K_1^{g_i}$, $l\in L_2.$
This proves the last claim.
\end{proof}

\subsection{Generalized Steinberg representations}\label{Steinberg}
We fix a Borel subgroup $B$ of the base change $G_{F_\p}$ of $G$ to $F_\p$ and a maximal split torus $T\subseteq B$.
Let $\Delta$ be the associated root basis.
For any subset $I\subseteq\Delta$ let $B\subseteq P_I\subseteq G_{F_\p}$ be the associated standard parabolic subgroup.
For a ring $R$ let $i^{\infty}_I(R)=C^{\infty}(P_I\backslash G_\p,R)$ be the smooth $G_\p$-representation of locally constant $R$-valued functions on the quotient $P_I\backslash G_\p$.
The \emph{generalized (smooth) Steinberg representation} associated to $I\subseteq \Delta$ is given by the quotient
$$v^{\infty}_I(R)= i^{\infty}_I(R) / \sum_{I\subset J\subset \Delta, I\neq J} i^{\infty}_J(R).$$
If $I=\left\{i\right\}$ consists of a single element, we put $v^{\infty}_i(R)=v^{\infty}_{\left\{i\right\}}(R).$
In case $I=\emptyset$ the representation $\St^{\infty}_{G_\p}(R)=v^{\infty}_\emptyset(R)$ is the usual $R$-valued Steinberg representation of $G_\p$.
If $R=\Omega$ is a field of characteristic zero, the generalized Steinberg representations $v^{\infty}_I(\Omega)$ are known to be pairwise non-isomorphic, irreducible representations.
The constituents of the Jordan--Hölder series of $i^{\infty}_I(\Omega)$ are exactly those $v^{\infty}_{I^\prime}(\Omega)$ with $I\subseteq I^{\prime}$, each occurring with multiplicity one.

\begin{Thm}[Borel--Serre]\label{resolution1}
The integral Steinberg representation $\St^{\infty}_{G_\p}(\Z)$ is flawless.
More precisely, there exists a finite length resolution of smooth $G_\p$-modules
$$0\too P_d\too\cdots\too P_0\too\St^{\infty}_{G_\p}(\Z)\too 0$$
with the following properties:
\begin{enumerate}[(a)]
\item\label{higher}
For $i\geq 1$ the representation $P_i$ is isomorphic to a finite direct sum of representations of the form
$$\cind_{K_\p}^{G_\p}\Z(\chi)$$
with $K_\p\subseteq G_\p$ compact open subgroups such that
\begin{itemize}
\item $K_\p$ has a normal subgroup that is a parahoric subgroup but not an Iwahori subgroup and 
\item $\chi\colon K_\p\to \left\{\pm1\right\}$ is a character that is trivial on said parahoric subgroup  
\end{itemize}
\item\label{zeroth}
The representation $P_0$ is isomorphic to
$$\cind_{K_\p}^{G_\p}\Z(\chi)$$
with $K_\p\subseteq G_\p$ a compact open subgroup such that
\begin{itemize}
\item $K_\p$ has a normal subgroup that is an Iwahori subgroup and
\item $\chi\colon K_p\to \left\{\pm1\right\}$ is a character that is trivial on said Iwahori subgroup  
\end{itemize}
\end{enumerate}
\end{Thm}
\begin{proof}
By Theorem 5.6 of \cite{BS2} the cohomology of the Bruhat--Tits building associated to $G_\p$ vanishes outside the top degree and the top degree cohomology is isomorphic to the integral Steinberg representation.
Writing down the simplicial complex gives the sought-after resolution.

For \eqref{zeroth}: Note that all Iwahori subgroups are in fact conjugated.
\end{proof}

The following result is essentially due to Schneider and Stuhler.
\begin{Thm}[Schneider--Stuhler]\label{genflaw}
Suppose that every simple factor of the split semi-simple group $G_{F_\p}$ is of type $A_{n}$.
Then $v^{\infty}_I(\Z)$ is flawless for all $I\subseteq\Delta$.
\end{Thm}
\begin{proof}
By Lemma \ref{flawlessproperties} \eqref{flawlesspullback} we may assume that $G_{F_\p}$ is of adjoint type, i.e.~we have an isomorphism
$$G_{F_p}\cong \prod_{i=1}^{k}PGL_{n_i,F_\p}$$
for some $n_i\geq 1$.
(Note that we always assume that $G_{F_\p}$ is split.)
Thus, every generalized Steinberg representation $v^{\infty}_I(\Z)$ of $G_\p$ can be written as a tensor product
$$v_I(\Z)=\bigotimes_{i=1}^{k} v^{\infty}_{I_i}(\Z)$$
of generalized Steinberg representations $v^{\infty}_{I_i}(\Z)$ for the groups $PGL_{n_i}(F_\p).$

By Lemma \ref{flawlessproperties} \eqref{flawlesstensor} we may reduce to the case $G_\p=PGL_{n}(F_\p),$
which is Theorem 8 of \cite{SScoh}.
\end{proof}

\begin{Rem}\label{localg}
\begin{enumerate}\thmenumhspace
\item 
By \cite{GK2}, Theorem 1, we may embed $v_I(\Z)$ into the space of continuous functions from an Iwahori subgroup of $G_\p$ into a certain finitely generated, free $\Z$-module.
As the later is free as a $\Z$-module, so is the generalized Steinberg representation $v_I(\Z)$.
It is also known that $v_I(\Z)$ is finitely generated over $\Z[G_\p]$.
More precisely, $v_I(\Z)$ is a cyclic $\Z[G_\p]$-module (see \cite{Ait}, Theorem 3.4)
\item Similarly as in the smooth case, one could define flawlessness for locally algebraic representations with coefficients in a finite extension of $F_\p$ or $\mathcal{O}_\p$.
In order to generalize the constructions of this paper to automorphic representations, which are cohomological with respect to an arbitrary coefficient system, one needs the existence of flawless integral lattices in twists of the Steinberg representation by irreducible algebraic representations.
This seems to be a hard problem in general.  
In case $G=PGL_2$ the existence of such flawless lattices in the relevant cases is known by a result of Vignéras (see \cite{Vi}, Proposition 0.9).
\end{enumerate}
\end{Rem}

\subsection{Smooth extensions of generalized Steinberg representations}\label{Extensions1}
We recall the theorem of Dat and Orlik about extensions of smooth generalized Steinberg representations and deduce a simple but crucial corollary.
For this section we fix a field $\Omega$ of characteristic $0$ and often abbreviate $v^{\infty}_I=v^{\infty}_I(\Omega)$ respectively $\St^{\infty}_{G_\p}=\St^{\infty}_{G_\p}(\Omega)$.
The category $\mathfrak{C}^{\sm}_{\Omega}(G_\p)$ has enough injective and projective objects.
For $V,W\in \ob(\mathfrak{C}^{\sm}_{\Omega}(G_\p))$ we abbreviate
$$\Ext^{i}_{\sm}(V,W)=\Ext^{i}_{\mathfrak{C}^{\sm}_{\Omega}(G_\p)}(V,W).$$
\begin{Thm}[Dat, Orlik]\label{DatOrlik}
Given two subsets $I,J\subseteq\Delta$ we put $\delta(I,J)=|I\cup J|-|I\cap J|.$
\begin{enumerate}[(i)]
\item\label{DO1} For $I,J\subseteq \Delta$ we have
$$\Ext^{i}_{\sm}(v^{\infty}_I(\Omega),v^{\infty}_J(\Omega))\cong
\begin{cases}
\Omega & \mbox{if } i=\delta(I,J)\\
0 & \mbox{else.}
\end{cases}$$
\item\label{DO2} Let $I,J,K\subseteq\Delta$ be three subsets with $\delta(I,J)+\delta(J,K)=\delta(I,K)$.
Then, the cup product map
$$\Ext^{\delta(I,J)}_{\sm}(v^{\infty}_I,v^{\infty}_J) \otimes \Ext^{\delta(J,K)}_{\sm}(v^{\infty}_J,v^{\infty}_K)
\xrightarrow{\cup} \Ext^{\delta(I,K)}_{\sm}(v^{\infty}_I,v^{\infty}_K)$$
is an isomorphism.
\end{enumerate}
\end{Thm}
\begin{proof}
The first claim is proven by Dat (cf.~Theorem 1.3 of \cite{Dat}) and independently by Orlik (cf.~Theorem 1 of \cite{Orlik}).
The second claim is as result of Dat (see \textit{loc.cit.}).
\end{proof}

\begin{Rem}\label{split}
The first claim also holds for non-split semi-simple groups (see \cite{Orlik}).
It is likely that one can refine the arguments in \textit{loc.cit.~}to prove the second claim in this generality.
But note that the split case is the most interesting one for our purposes.
On the contrary, if $G_\p$ is compact, e.g.~the group of reduced norm one elements of the non-split quaternion algebra over $F_\p$, then the only generalized Steinberg representation is the trivial one-dimensional representation and the category $\mathfrak{C}^{\sm}_{\Omega}(G_\p)$ is semi-simple.
Thus, there are no interesting smooth extension classes.
\end{Rem}

\begin{Def}
Let $I,J\subseteq\Delta$ be two subsets with $\delta(I,J)=1$.
We define
$$0 \too v^{\infty}_J(\Omega) \too \mathcal{E}_{I,J}^{\infty}(\Omega) \too v^{\infty}_I(\Omega) \too 0$$
to be any generator of the one-dimensional space $\Ext^{1}_{\sm}(v^{\infty}_I,v^{\infty}_J)$.
\end{Def}

\begin{Cor}\label{extvanishing}
For all subsets $J\subseteq I\subseteq\Delta$ with $|I|=|J|+1$ we have:
$$\Ext_{\sm}^{d}(\mathcal{E}_{I,J}^{\infty}(\Omega),\St^{\infty}_{G_\p}(\Omega))=0\quad \forall d\geq 0.$$
\end{Cor}
\begin{proof}
By definition we have a short exact sequence of the form
$$0\too v^{\infty}_J\too\mathcal{E}_{I,J}^{\infty}(\Omega)\too v^{\infty}_I\too 0.$$
Most of the terms of the long exact sequence induced by applying $\Hom(\cdot,\St^{\infty}_{G_\p})$ vanish by Theorem \ref{DatOrlik} \eqref{DO1}.
The remaining terms are
$$0\too \Ext_{\sm}^{|J|}(\mathcal{E}_{I,J}^{\infty}(\Omega),\St^{\infty}_{G_\p})\too \Ext_{\sm}^{|J|}(v^{\infty}_J,\St^{\infty}_{G_\p})\xrightarrow{\partial} \Ext_{\sm}^{|I|}(v^{\infty}_I,\St^{\infty}_{G_\p})\too 0.$$
The claim follows since both, the middle term and the term on the right, are one-dimensional by \ref{DatOrlik} \eqref{DO1}.
\end{proof}

\subsection{Locally analytic extensions}\label{Extensions2}
We compute spaces of extensions of generalized Steinberg representations and their locally analytic counterparts.
This slightly generalizes results of Ding (see \cite{Ding}, Section 2.2).

We fix a finite extension $\Omega$ of $\Q_p$.
Let $D(G_\p)=D(G_\p,\Omega)$ be the algebra of $\Omega$-valued distributions on $G_\p$ in the sense of Schneider and Teitelbaum (cf.~\cite{ST}) and $\mathcal{M}(G_\p)$ the abelian category of abstract $D(G_\p)$-modules.
If $V$ is an $(\Q_p-)$analytic representations with coefficients in $\Omega$, its continuous dual $V^{\vee}$ is naturally a $D(G_\p)$-module.
Given two admissible locally analytic representation $V$ and $W$ (in the sense of \cite{STadm}, Section 6) we put
$$\Ext^i_{\an}(V,W)=\Ext^i_{\mathcal{M}(G_\p)}(W^{\vee},V^{\vee}).$$
Note that in general the relation between $\Ext^i_{\an}(V,W)$ and the group of Yoneda $i$-extensions of $W$ by $V$ in the  category of admissible locally analytic representations of $G_\p$ is not clear.
The following two specials cases are known and suffice for our purposes:
First, by Theorem 6.3 of \cite{STadm}, the space $\Ext^0_{\an}(V,W)$ agrees with the space of $G_\p$-equivariant continuous homomorphism from $V$ to $W$. 
Second, the group $\Ext^1_{\an}(V,W)$ is isomorphic to the group of admissible locally analytic extensions of $V$ by $W$ (see \cite{Br3}, Lemma 2.1.1). 

Let $I\subseteq\Delta$ be a subset and $\tau$ an admissible locally analytic representation of $P_I$.
The \emph{locally analytic induction} of $\tau$ to $G_\p$ is the space of functions
$$\II_{P_I}^{\an}(\tau)=\left\{f\colon G_\p\to \tau \mbox{ locally analytic}\mid f(pg)=p.f(g)\ \forall p\in P_I,\ g\in G_\p \right\}.$$
It is well known that the locally analytic induction of $\tau$ is again an admissible locally analytic representation (see for example \cite{EmJac2}, Proposition 2.1.2).
The \emph{generalized locally analytic Steinberg representation} with respect to $I$ is defined as the quotient
$$\VV^{\an}_I(\Omega)=\II_{P_I}^{\an}(\Omega)/\sum_{I\subset J\subset \Delta, I\neq J} \II^{\an}_{P_J}(\Omega).$$
As before, we put $\St^{\an}_{G_\p}(\Omega)=\VV^{\an}_\emptyset(\Omega)$ and often drop the field of coefficients $\Omega$ from the notation.

For $I\subseteq \Delta$ we let $L_I\subseteq P_I$ be the Levi subgroup containing our fixed maximal torus and $Z_I\subseteq L_I$ its centre.
\begin{Pro}\label{DingCW}
Let $I$ and $J$ be subsets of $\Delta$. 
We have
\begin{align*}
\Ext^{i}_{\an}(i^{\infty}_{I}(\Omega),\II^{\an}_{P_J}(\Omega))\cong
\begin{cases}
\HH^{i}_{an}(L_J,\Omega) & \mbox{if}\ J\subseteq I,\\
0 & \mbox{else,}
\end{cases}
\end{align*}
where $\HH^{i}_{\an}(L_I,\Omega)$ denotes the $i$-th analytic cohomology group as defined by Casselman and Wigner in \cite{CaWi}.
\end{Pro}
\begin{proof}
For the group $GL_n(F_\p)$ the claim is proven in \cite{Ding}, Proposition 2.11.
The proof of \textit{loc.cit.~}works verbatim in our situation. 
\end{proof}

Let us make the above isomorphism more explicit in the case $i=1$, $J\subseteq I$.
As usual we can identify $\HH^{1}_{\an}(L_J,\Omega)$ with the space of analytic homomorphism from $L_J$ to $\Omega$.
Since $L_J$ is isogenous to its centre times a semi-simple group, we see that the restriction map
$$\Hom_{\cont}(L_J,\Omega)\too \Hom_{\cont}(Z_J,\Omega)$$
is an isomorphism.
In particular, every continuous homomorphism from $L_J$ to $\Omega$ is already $\Q_p$-analytic.
Thus, we get an isomorphism
$$\HH^{1}_{\an}(L_J,\Omega)\cong \Hom_{\cont}(L_J,\Omega).$$
Moreover, the inclusion $L_J\subseteq P_I$ induces an isomorphism 
$$\Hom_{\cont}(P_J,\Omega)\xlongrightarrow{\cong} \Hom_{\cont}(L_J,\Omega).$$
Thus, by Proposition \ref{DingCW} we get an isomorphism
\begin{align}\label{DingInd}
\Hom_{\cont}(P_I,\Omega)\xlongrightarrow{\cong} \Ext^{1}_{\an}(i^{\infty}_{I}(\Omega),\II^{\an}_{P_J}(\Omega)).
\end{align}

Given a continuous homomorphism $\lambda\in\Hom_{\cont}(P_J,\Omega)$ we let $\tau_\lambda$ be the two-dimensional representation of $P_J$ given by
$$\tau_\lambda(p)=\begin{pmatrix} 1 & \lambda(p)\\ 0 & 1\end{pmatrix}.$$
By definition we have a sequence of the form
$$0\too \II_{P_J}^{\an}(\Omega)\too \II_{P_J}^{\an}(\tau_\lambda)\too \II_{P_J}^{\an}(\Omega)\too 0,$$
which is exact by \cite{Ko}, Proposition 5.1 and Remark 5.4.
Pullback along the inclusion $i^{\infty}_{I}(\Omega)\to \II_{P_J}^{\an}(\Omega)$ therefore gives a class
$$\widetilde{\mathcal{E}}^{\an}_{I,J}(\lambda)\in \Ext^{1}_{\an}(i^{\infty}_{I}(\Omega),\II^{\an}_J).$$
It can be checked as in \cite{Ding}, Lemma 2.13, that there exists a non-zero constant $c\in \Omega^{\times}$ (independent of $\lambda$) such that $\lambda$ is mapped to $\widetilde{\mathcal{E}}^{\an}_{I,J}(\lambda)$ under the isomorphism \eqref{DingInd}.

Now suppose that $J$ is strictly contained in $I$.
Let $\mathcal{E}^{\an}_{I,J}(\lambda)\in \Ext^{1}_{\an}(i^{\infty}_{I}(\Omega),\VV^{\an}_J)$
be the pushforward of $\widetilde{\mathcal{E}}^{\an}_{I,J}(\lambda)$ along the quotient map $\II_{P_J}^{\an}(\Omega)\to \VV^{\an}_J$.

The inclusion $P_J\into P_I$ induces an injection
$$\Hom_{\cont}(P_I,\Omega)\intoo\Hom_{\cont}(P_J,\Omega).$$
We will identify $\Hom_{\cont}(P_I,\Omega)$ with its image in $\Hom_{\cont}(P_J,\Omega)$.
\begin{Lem}\label{welldef}
The extension $\mathcal{E}^{\an}_{I,J}(\lambda)$ is split for all $\lambda\in\Hom_{\cont}(P_I,\Omega).$
\end{Lem}
\begin{proof}
Let $\lambda$ be an element of $\Hom_{\cont}(P_I,\Omega)$.
Then $\tau_\lambda$ can be extended to a representation of $P_I$.
The pullback of the exact sequence
$$0\too \II_{P_I}^{\an}(\Omega)\too \II_{P_I}^{\an}(\tau_\lambda)\too \II_{P_I}^{\an}(\Omega)\too 0$$
along the inclusion $i^{\infty}_{I}(\Omega)\to \II_{P_I}^{\an}(\Omega)$ gives a class $\widetilde{\mathcal{E}}(\lambda)\in \Ext^{1}_{\an}(i^{\infty}_{I}(\Omega),\II^{\an}_{P_I}(\Omega))$.
By construction $\widetilde{\mathcal{E}}(\lambda)$ is mapped to $\mathcal{E}^{\an}_{I,J}(\lambda)$ under the composition $\II^{\an}_{P_I}(\Omega)\to\II^{\an}_{P_J}(\Omega)\to\VV^{\an}_J$, which is the zero map by definition. 
\end{proof}

Suppose that $J\subseteq I$ with $|I|=|J|+1$ and let $i$ be the unique root in $I$ which is not contained in $J$.
Then, pullback via $i$ defines an injection
$$\Hom_{\cont}(F_\p^{\times},\Omega)\too \Hom_{\cont}(Z_J,\Omega)\cong\Hom_{\cont}(P_J,\Omega),\quad \psi\mapstoo \psi\circ i.$$
Its image is a complement to $\Hom_{\cont}(P_I,\Omega)$.

\begin{Cor}\label{extisom}
Let $J\subseteq I$ be subsets of $\Delta$ with $|I|=|J|+1.$
The homomorphism
$$\Hom_{\cont}(F_\p^{\times},\Omega)\too \Ext^{1}_{\an}(i^{\infty}_{I}(\Omega),\VV^{\an}_J(\Omega)),\quad \lambda\mapstoo \mathcal{E}^{\an}_{I,J}(\lambda\circ i)$$
is an isomorphism.
\end{Cor}
\begin{proof}
This is a slight generalization of of \cite{Ding}, Corollary 2.17.
For the convenience of the reader we include a proof of the claim.

By the discussion preceding the proposition it is enough to show that the map
$$\Hom_{\cont}(P_J,\Omega)/\Hom_{\cont}(P_I,\Omega)\too \Ext^{1}_{\an}(i^{\infty}_{I}(\Omega),\VV^{\an}_J(\Omega))$$
is an isomorphism.

By \cite{OrSchraen}, Theorem 4.2, the projection $\II_{P_J}^{\an}(\Omega)\onto \VV^{\an}_J(\Omega)$ extends to a long exact sequence
\[
0\hspace{-0.1em}\to \hspace{-0.1em}\II_{G}^{\an}(\Omega)\hspace{-0.1em}
\to  \hspace{-0.1em}\bigoplus_{\substack{J\subseteq K\subseteq \Delta\\ |\Delta\setminus K|=1}} \II_{P_K}^{\an}(\Omega)\hspace{-0.1em}
\to \hspace{-0.1em}\cdots \hspace{-0.1em}
\to  \hspace{-0.1em}\bigoplus_{\substack{J\subseteq K\subseteq \Delta\\ |K\setminus I|=1}} \II_{P_K}^{\an}(\Omega)\hspace{-0.1em}
\to \hspace{-0.1em}\II_{P_J}^{\an}(\Omega)\to \VV^{\an}_J(\Omega) \hspace{-0.1em}\to \hspace{-0.1em} 0.
\]
Note that here we use our assumption that the prime $p$ is bigger than $2$ (resp.~bigger than $3$) if the root system of $G$ has irreducible components of type $B$, $C$ or $F_4$ (resp.~of type $G_2$).
Let
$$
E_1^{-s,r}=\bigoplus_{\substack{J\subseteq K\subseteq\Delta \\ |K\setminus J|=s}}\Ext^{r}_{\an}(i^{\infty}_{I}(\Omega),\II^{\an}_K(\Omega))
\Rightarrow \Ext^{r-s}_{\an}(i^{\infty}_{I}(\Omega),\VV^{\an}_J(\Omega))
$$
be the induced spectral sequence.
By Proposition \ref{DingCW} all but the $0$-th and $(-1)$-th columns of the $E_1$-page vanish. 

The first row is given by the map
$$
d_1^{1,-1}\colon \Ext^{1}_{\an}(i^{\infty}_{I}(\Omega),\II^{\an}_I(\Omega))\too \Ext^{1}_{\an}(i^{\infty}_{I}(\Omega),\II^{\an}_J(\Omega)),
$$
which can be identified via the isomorphism \eqref{DingInd} with the inclusion
$$\Hom_{\cont}(P_I,\Omega) \into \Hom_{\cont}(P_J,\Omega).$$

Thus, it is enough to prove that the second row map
$$
d_1^{1,-2}\colon \HH^2_{\an}(L_I,\Omega)\too \HH^2_{\an}(L_J,\Omega)
$$
is injective.
By \cite{Schraen}, Corollary 3.14, (see also \cite{Ding}, Remark 2.12) we can identify
\begin{align*}
\HH^2_{\an}(L_K,\Omega)\cong\wedge^2 \Hom_{\cont}(Z_K(L), \Omega)
\end{align*}
for every $K\subseteq \Omega$ and the map $d_1^{1,-2}$ with the canonical inclusion
$$
\wedge^2 \Hom_{\cont}(Z_I(L), \Omega)\into \wedge^2 \Hom_{\cont}(Z_J(L), \Omega).
$$
Thus, the claim follows.
\end{proof}

The following theorem was proven by Ding in the case $G=PGL_n$ and $J=\emptyset$ by computing dimensions on both sides (see \cite{Ding}, Theorem 2.19).
Our approach is slightly different.
We show explicitly that the extensions $\mathcal{E}^{\an}_{I,J}(\lambda)$ lift to an element in $\Ext^{1}_{\an}(v^{\infty}_{I},\VV^{\an}_J)$.
\begin{Thm}\label{Ding}
Let $J\subseteq I \subseteq\Delta$ be subsets with $|I|=|J|+1$.
The natural map
$$\Ext^{1}_{\an}(v^{\infty}_{I},\VV^{\an}_J)\too \Ext^{1}_{\an}(i^{\infty}_{I}(\Omega),\VV^{\an}_J)$$
is an isomorphism.
In particular, there exits a canonical isomorphism
$$\Hom_{\cont}(F_\p^{\times},\Omega)\xlongrightarrow{\cong}\Ext^{1}_{\an}(v^{\infty}_{I},\VV^{\an}_J).$$
\end{Thm}
\begin{proof}
First let us show that the map is injective.
Let $w^{\infty}_{I}$ be the kernel of the map $i^{\infty}_{I}(\Omega)\to v^{\infty}_{I}$.
Considering the long exact sequence of $\Ext$-groups it is enough to show that
$$\Hom_{G_\p}(w^{\infty}_{I},\VV^{\an}_J)=0.$$
By the main theorem of \cite{OrSchraen} the only smooth Jordan--Hölder factor of $\VV^{\an}_J$ is the smooth subrepresentation $v_J^\infty(\Omega).$
On the other hand, the Jordan--Hölder factors of $w^{\infty}_{I}$ are given by the smooth generalized Steinberg representations $v^{\infty}_{K}(\Omega)$ with $I\subsetneq K.$
Thus, the claim follows.

In order to show surjectivity let us first note that by a similar argument the natural map
$$\Ext^{1}_{\an}(w^{\infty}_{I},\VV^{\an}_J)\too \bigoplus_{k\in \Delta\setminus I}\Ext^{1}_{\an}(i^{\infty}_{I\cup\left\{k\right\}}(\Omega),\VV^{\an}_J)$$ is injective.
Therefore, by Corollary \ref{extisom} it is enough to show the following: Let $\lambda$ be an element of $\Hom_{\cont}(P_J,\Omega)$. Then the pullback of the extension $\mathcal{E}^{\an}_{I,J}(\lambda)$ along the inclusion $i^{\infty}_{I\cup\left\{k\right\}}(\Omega)\into i^{\infty}_{I}(\Omega)$ is trivial for all $k\in \Delta\setminus I$.

Given $k\in \Delta\setminus I$ there exists a homomorphism
$$\lambda_k\in\Hom_{\cont}(P_{J\cup\left\{k\right\}},\Omega)\subseteq \Hom_{\cont}(P_{J},\Omega)$$
such that its image in $\Hom_{\cont}(P_J,\Omega)/\Hom_{\cont}(P_I,\Omega)$ agrees with $\lambda$.
By Lemma \ref{welldef} the extensions $\mathcal{E}^{\an}_{I,J}(\lambda)$ and $\mathcal{E}^{\an}_{I,J}(\lambda_k)$ agree.
The pullback of $\mathcal{E}^{\an}_{I,J}(\lambda_k)$ along the inclusion $i^{\infty}_{I\cup\left\{k\right\}}(\Omega)\into i^{\infty}_{I}(\Omega)$ is equal to to the pullback of $\mathcal{E}^{\an}_{J\cup\left\{k\right\},J}(\lambda_k)$ along the inclusion $i^{\infty}_{I\cup\left\{k\right\}}(\Omega)\into i^{\infty}_{J\cup\left\{k\right\}}(\Omega)$.
Since the extension $\mathcal{E}^{\an}_{J\cup\left\{k\right\},J}(\lambda_k)$ is split by Lemma \ref{welldef}, the claim follows.
\end{proof}

\begin{Rem}\begin{enumerate}[(i)]\thmenumhspace
\item Let $\lambda\in \Hom_{\cont}(F_\p^{\times},\Omega)$ be a smooth character.
It is easy to see that the class $\mathcal{E}^{\an}_{I,J}(\lambda\circ i)$ is in the image of the natural injection
$$\Ext_{\sm}^{1}(v^{\infty}_{I},v^{\infty}_J)\to \Ext^{1}_{\an}(v^{\infty}_{I},\VV^{\an}_J).$$
Therefore, if $\lambda$ is non-zero, then $\mathcal{E}^{\an}_{I,J}(\lambda \circ i)$ is a non-zero multiple of $\mathcal{E}_{I,J}^{\infty}.$
\item The calculations are valid for arbitrary split reductive groups, i.e.~we do not have to assume semi-simplicity.
As in \cite{Ding}, one could also allow twists by a fixed irreducible, algebraic representation.
\end{enumerate}
\end{Rem}

\subsection{Extensions of Banach representations}\label{Extensions3}
In this section we give analogues of the results of the previous section in the realm of Banach representations.
In contrast to the locally analytic setting it is crucial that we work with trivial coefficients.

As before let $\Omega$ be a finite extension of $\Q_p$.
Let $R$ be its its ring of integers with uniformizer $\varpi$.
We denote by $\mathcal{U}_{\Omega}(G_\p)$ the category of  unitary $\Omega$-Banach space representation of $G_\p$ that are admissible in the sense of \cite{EmAn}, Definition 6.2.3.
By \cite{EmAn}, Corollary 6.2.16, $\mathcal{U}_{\Omega}(G_\p)$ is an abelian category.
For $V$ and $W$ admissible Banach representations we write
$$\Ext^i_{\cont}(V,W)=\Ext^i_{\mathcal{U}_{\Omega}(G_\p)}(V,W)$$
for the space of Yoneda $i$-Extensions of $W$ by $V$ in $\mathcal{U}_{\Omega}(G_\p)$.

Given a subset $I\subseteq \Delta$ and an admissible Banach space representation $\tau$ of $P_I$ we define its \emph{continuous induction} to $G_\p$ as the space of functions
$$\II_{P_I}^{\cont}(\tau)=\left\{f\colon G_\p\to \tau \mbox{ continuous}\mid f(pg)=p.f(g)\ \forall p\in P_I,\ g\in G_\p \right\}.$$
This is an admissible Banach space representation of $G_\p$ by \cite{EmOrd1}, Proposition 4.1.7.
We define the \emph{generalized continuous Steinberg representations} with respect to $I$ via the quotient
$$\VV^{\cont}_I(\Omega)=\II_{P_I}^{\cont}(\Omega)/\sum_{I\subset J\subset \Delta, I\neq J} \II^{\cont}_{P_J}(\Omega)$$
and put $\St^{\cont}_{G_\p}(\Omega)=\VV^{\cont}_\emptyset(\Omega)$.
As before, we often drop $\Omega$ from the notation.

Further, we define the integral version $\II^{\cont}_{P_I}(R)$ (respectively $\VV^{\cont}_I(R)$) as the $\pi$-adic completion of $i^{\infty}_I(R)$ (respectively $v^{\infty}_I(R)$), which both are objects in $\mathcal{U}_{R}(G_\p)$.

\begin{Lem}
The representation $\II^{\cont}_{P_I}(R)$ defines an open lattice in $\II^{\cont}_{P_I}(\Omega)$ and $\II^{\cont}_{P_I}(\Omega)$ is the universal unitary completion of both, $i_I^{\infty}(\Omega)$ and $\II^{\an}_{P_I}(\Omega)$.
\end{Lem}
\begin{proof}
The first statement follows directly from the definition.
Since $i^{\infty}_I(R)$ is finitely generated as an $R[G_\p]$-module the claim about universal unitary completions in the smooth case then follows from \cite{EmUnitary}, Proposition 1.17.
In the locally analytic case the claim about universal unitary completions is a special case of \cite{BrHe}, Proposition 3.1
\end{proof}

\begin{Lem}\label{quotients}
Suppose that
$$A\too B \too C\too 0$$
is a strict exact sequence of continuous $G_\p$-representations on locally convex $\Omega$-vector spaces.
Suppose that $A$ and $B$ admit universal unitary completions $A^{\univ}$ and $B^{\univ}$ that are admissible.
Then the cokernel 
$$C^{\univ}=\Coker(A^{\univ} \to B^{\univ})$$ is the universal unitary completion of $C$.
\end{Lem}
\begin{proof}
By definition the canonical map $B\to C^{\univ}$ factors over $C$.
Since the image of $B$ in $B^{\univ}$ is dense, so is the image of $C$ in $C^{\univ}.$

Let $T$ be a unitary representation of $G_\p$ and $f\colon C\to T$ a continuous $G_\p$-equivariant map.
By universality of $B^{\univ}$ the map $B\to C\xrightarrow{f} T$ uniquely factors through the map $B\to B^{\univ}.$
The composition $A\to B\to C\xrightarrow{f} T$ is the zero map.
Thus, by universality of $A^{\univ}$ the map $A^{\univ}\to B^{\univ}\to T$ is the zero map.
In other words, the map $B^{\univ}\to T$ factors through the quotient map $B^{\univ}\onto C^{\univ}.$

In conclusion, we have constructed a $G_\p$-equivariant continuous homomorphism $f^{\univ}\colon C^{\univ}\to T$ such that the diagram
\begin{center}
\begin{tikzpicture}
    \path 	(0,0) 	node[name=A]{$C$}
		(3,-1.5) 	node[name=B]{$C^{\univ}$}
		(3,0) 	node[name=C]{$T$};
    \draw[->] (A) -- (C) node[midway, above]{$f$};
    \draw[->] (B) -- (C) node[midway, right]{$f^{\univ}$};
		\draw[->] (A) -- (B) ;
\end{tikzpicture} 
\end{center}
 commutes. 
The map $f^{\univ}$ is uniquely determined since the image of $C$ in $C^{\univ}$ is dense.
\end{proof}

From the two lemmas above we immediately deduce the following corollary.
\begin{Cor}\label{universal}
The representation $\VV^{\cont}_{I}(R)$ defines an open lattice in $\VV^{\cont}_{I}(\Omega)$ and $\VV^{\cont}_{I}(\Omega)$
is the universal unitary completion of both, $v_I^{\infty}(\Omega)$ and $\VV^{\an}_{I}(\Omega)$.
\end{Cor}

As before, let us fix subsets $J\subseteq I\subseteq\Delta$ with $|I|=|J|+1$.
Like in the locally analytic case each homomorphism $\lambda\in\Hom_{\cont}(P_J,\Omega)$ induces an sequence
$$0\too \II_{P_J}^{\cont}(\Omega)\too \II_{P_J}^{\cont}(\tau_\lambda)\too \II_{P_J}^{\cont}(\Omega)\too 0,$$
which is exact by \cite{EmOrd1}, Proposition 4.1.5.
Pushforward via $\II_{P_J}^{\cont}(\Omega)\to \VV^{\cont}_J$ and pullback via $\II_{P_I}^{\cont}(\Omega)\to \II_{P_J}^{\cont}(\Omega)$
gives a class $$\mathcal{E}^{\cont}_{I,J}(\lambda)\in \Ext^{1}_{\cont}(\II^{\cont}_{P_I}(\Omega),\VV^{\cont}_J).$$

\begin{Pro}\label{contpro}
Let $J\subseteq I\subseteq\Delta$ be subsets with $|I|=|J|+1$.
\begin{enumerate}[(i)]
\item The extension $\mathcal{E}^{\cont}_{I,J}(\lambda)$ is split for all $\lambda\in\Hom_{\cont}(P_I,\Omega).$
\item For every $\lambda\in\Hom_{\cont}(P_J,\Omega)$ the extension $\mathcal{E}^{\cont}_{I,J}(\lambda)$ lies in the image of the canonical injective map
$$\Ext^{1}_{\cont}(\VV^{\cont}_{I},\VV^{\cont}_J)\too \Ext^{1}_{\cont}(\II^{\cont}_{P_I}(\Omega),\VV^{\cont}_J).$$
\item Let $i\in I$ be the unique root which is not contained in $J$.
The map
\begin{align}\label{context}
\Hom_{\cont}(F_\p^{\times},\Omega)\too\Ext^{1}_{\cont}(\VV^{\cont}_{I},\VV^{\cont}_J),\quad \lambda\mapstoo \mathcal{E}^{\cont}_{I,J}(\lambda\circ i)
\end{align}
is an isomorphism.
\end{enumerate}
\end{Pro}
\begin{proof}
The first claim can be proven exactly as Lemma \ref{welldef}.

Let $\WW^{\cont}_{I}$ be the kernel of the map $\II^{\cont}_{P_I}(\Omega)\to \VV^{\cont}_{I}$ and let $w^{\infty}_{I}$ be the kernel of the map $i^{\infty}_{I}(\Omega)\to v^{\infty}_{I}$.
The subspace $w^{\infty}_{I} \subseteq \WW^{\cont}_{I}$ is dense and, thus, we see that
$$\Hom_{\mathcal{U}_{\Omega}(G_\p)}(\WW^{\cont}_{I}, \VV^{\cont}_{J})\subseteq \Hom_{G_\p}(w^{\infty}_{I},\VV^{\cont}_{J})=\Hom_{G_\p}(w^{\infty}_{I},\VV^{\an}_{J})=0$$
as in the proof of Theorem \ref{Ding}.
Hence, by the long exact sequence of $\Ext$-groups the map
$$\Ext^{1}_{\cont}(\VV^{\cont}_{I},\VV^{\cont}_J)\too \Ext^{1}_{\cont}(\II^{\cont}_{P_I}(\Omega),\VV^{\cont}_J)$$
is injective.

Given an element $\lambda\in\Hom_{\cont}(P_I,\Omega)$ we have to show that the pullback of the exact sequence
$$0 \too \VV^{\cont}_J\too \mathcal{E}^{\cont}_{I,J}(\lambda)\too\II^{\cont}_{P_I}(\Omega) \too 0$$
along the inclusion $\WW^{\cont}_{I} \into \II^{\cont}_{I}$ is split.
One may deduce from lemma \ref{quotients} that $w^{\infty}_{I} \to \WW^{\cont}_{I}$ is a universal unitary completion
Thus, it is enough to prove that the pullback to $w^{\infty}_{I}$ is split.
But this follows from the proof of Theorem \ref{Ding}.

Passing to locally analytic vectors is an exact functor on the category of admissible Banach representations by Theorem 7.1 of \cite{STadm}.
Thus, we get a map
$$\Ext^{1}_{\cont}(\VV^{\cont}_{I},\VV^{\cont}_J)\too \Ext^{1}_{\an}(\VV^{\an}_{I},\VV^{\an}_J).$$
Since $\VV^{\cont}_{I}$ is the universal unitary completion of $v^{\infty}_{I}$ we see that the composition
$$\Ext^{1}_{\cont}(\VV^{\cont}_{I},\VV^{\cont}_J)\too \Ext^{1}_{\an}(\VV^{\an}_{I},\VV^{\an}_J)\too \Ext^{1}_{\an}(v^{\infty}_{I},\VV^{\an}_J)$$
is injective.
By definition it maps the extension $\mathcal{E}^{\cont}_{I,J}(\lambda)$ to its locally analytic counterpart $\mathcal{E}^{\an}_{I,J}(\lambda)$.
Therefore, the last claim follows from Theorem \ref{Ding}.
\end{proof}

Extensions coming from different simple roots are in a sense independent.
To be more precise: let us fix another subset of simple roots $J\subseteq K\subseteq \Delta $ such that $|K|=|J|+1$ and $K\neq I$.
We denote by $k\in K$ the unique simple root not contained in $J$.
The following lemma follows from a lengthy but rather straightforward computation. 
\begin{Lem}\label{indext}
For every two homomorphisms $\lambda_i, \lambda_k\in \Hom_{\cont}(F_\p^{\times},\Omega)$ the equality
$$\mathcal{E}^{\cont}_{I\cup K,I}(\lambda_k\circ k)\cup\mathcal{E}^{\cont}_{I,J}(\lambda_i\circ i)= -\ \mathcal{E}^{\cont}_{I\cup K,K}(\lambda_i\circ i)\cup\mathcal{E}^{\cont}_{K,J}(\lambda_k\circ k)$$
holds in $\Ext^2_{\cont}(\VV^{\cont}_{I\cup K},\VV^{\cont}_J).$
\end{Lem}

\subsection{Integral and mod $p$ extensions}\label{Extensions4}
Keeping the same notations as in the previous section we define $\mathcal{U}_{R}(G_\p)$ to be the abelian category of $\varpi$-adically admissible representations of $G_\p$ over $R$ (cf.~\cite{EmOrd1}).
For $V_0,W_0\in \mathcal{U}_{R}(G_\p)$ we put
$$\Ext^i_{\cont}(V,W)=\Ext^i_{\mathcal{U}_{\R}(G_\p)}(V,W),$$
where the right hand side denotes the group of Yoneda $i$-Extensions of $W_0$ by $V_0$ in $\mathcal{U}_{R}(G_\p)$.

If $V_0\in \ob(\mathcal{U}_{R}(G_\p))$, then $V_0 \otimes_R \Omega$ is an an admissible unitary Banach representation.
Vice versa, if $V$ is an admissible unitary Banach representation, then every open $G_p$-stable $R$-lattice $V_0 \subseteq V$ is a $\varpi$-adically admissible representation.
We call such a lattice residually finite, if the representation $V_0 \otimes_r R/\varpi$ has finite length.

It follows from the definition that the canonical map
$$\Hom_{\mathcal{U}_{R}(G_\p)}(V_0,W_0)\otimes_R \Omega \xlongrightarrow{\cong} \Hom_{\mathcal{U}_{\Omega}(G_\p)}(V_0 \otimes_R \Omega,W_0 \otimes_R \Omega)$$
is an isomorphism for all $V_0, W_0\in \ob(\mathcal{U}_{R}(G_\p))$.

Moreover, since taking tensor product with $\Omega$ is exact, there is also a natural map
$$\loc\colon \Ext^{1}_{\cont}(V_0,W_0)\too \Ext^{1}_{\cont}(V_0 \otimes_R \Omega, W_0 \otimes_R \Omega).$$
The following result is proven in Appendix B of \cite{Hauseux}.
\begin{Pro}\label{AppHau}
Let $V$ and $W$ be admissible unitary Banach representations with open $G_\p$-equivariant lattices $V_0$ and $W_0$.
Suppose that $V_0$ is residually finite.
Then, the localization map
$$\loc\colon \Ext^{1}_{\cont}(V_0,W_0)\too \Ext^{1}_{\cont}(V, W)$$
is an isomorphism.
\end{Pro}

The representations $\II^{\cont}_{P_I}(R)$ and $\VV^{\cont}_I(R)$ are objects in $\mathcal{U}_{R}(G_\p)$.
They are residually finite by the main theorems of \cite{StLy}.

\begin{Lem}\label{CDHN}
Let $J\subseteq I\subseteq\Delta$ be subsets with $|I|=|J|+1$
There exists an integer $k\in \Z$ such that the map
$$\Hom_{\cont}(F_\p^{\times},\varpi^{k}R)\xrightarrow{\eqref{context}}\Ext^{1}_{\cont}(\VV^{\cont}_{I}(\Omega),\VV^{\cont}_J(\Omega))$$ canonically factors over the map
$$\loc\colon \Ext^{1}_{\cont}(\VV^{\cont}_{I}(R),\VV^{\cont}_J(R)) \too \Ext^{1}_{\cont}(\VV^{\cont}_{I}(\Omega),\VV^{\cont}_J(\Omega)).$$
\end{Lem}
\begin{proof}
Let us begin by showing that the canonical map
$$\Ext^{1}_{\cont}(\VV^{\cont}_{I}(R),\VV^{\cont}_J(R))\too \Ext^{1}_{\cont}(\II^{\cont}_{P_I}(R),\VV^{\cont}_J(R))$$
is injective.
Let $\WW^{\cont}_{I}(R)$ be the kernel of the homomorphism $\II^{\cont}_{P_I}(R)\to \VV^{\cont}_{P_I}(R)$.
Considering the long exact sequence of $\Ext$-groups we have to show that
$$\Hom_{\mathcal{U}_{R}(G_\p)}(\WW^{\cont}_{I}(R),\VV^{\cont}_J(R))=0.$$
Thus, reducing $\bmod\ \varpi$ it is enough to show that
$$\Hom_{G_\p}(w^{\infty}_{I}(R/\varpi),v^{\infty}_J(R/\varpi))=0,$$
where $w^{\infty}_{I}(R/\varpi)$ denotes the kernel of the map $i_I^\infty(R/\varpi)\to v_I^\infty(R/\varpi).$
By the main theorems of \cite{StLy} the representation $v^{\infty}_J(R/\varpi)$ is irreducible and none of the Jordan--Hölder factors of $w^{\infty}_{I}(R/\varpi)$ are isomorphic to it and, thus, the claim follows.

Similar as before, we can define an extension
$$0 \too \VV^{\cont}_J(R) \too \mathcal{E}^{\cont}_{I,J}(\lambda)\too \II^{\cont}_{P_I}(R) \too 0$$
for every homomorphism $\lambda\in\Hom_{\cont}(F_\p^{\times},R)$.
Ideally, we would like to show that the pullback of the extension along the inclusion $\WW^{\cont}_{I}(R) \into\II^{\cont}_{P_I}(R)$ is split.
By Proposition \ref{contpro} the pullback is split after inverting $p$.
Thus, by Proposition \ref{AppHau} the extension is split after restricting to $\varpi^{k}\WW^{\cont}_{I}(R)$ for some $k\in \Z$.
Since $\Hom_{\cont}(F_\p^{\times},R)$ is a finitely generated $R$-module, we may choose an integer $k$ that does not depend on the homomorphism $\lambda\in\Hom_{\cont}(F_\p^{\times},R)$.
\end{proof}

We are now able to construct extensions of mod $p$ generalized Steinberg representations.
Namely, we can reduce the resulting map 
$$\Hom_{\cont}(F_\p^{\times},\varpi^{k}R)\too \Ext^{1}_{\cont}(\VV^{\cont}_{I}(R),\VV^{\cont}_J(R))$$
from the lemma above modulo $\varpi^{r}$ for $r\geq 0$ to obtain a map
$$\Hom_{\cont}(F_\p^{\times},\varpi^{k}R/\varpi^{k+r}R)\too \Ext^{1}_{\mathfrak{C}^{\sm}_{R/\varpi^{r}}(G_\p)}(v^{\infty}_{I}(R/\varpi^{r}),v^{\infty}_J(R/\varpi^{r})).$$

\begin{Rem}
In \cite{CDHN} Colmez, Dospinescu, Hauseux and Nizioł construct an isomorphism from $\Ext^{1}_{\cont}(\VV^{\cont}_{I}(R),\VV^{\cont}_J(R))$ to $\Hom_{\cont}(F_\p^{\times},R)$. It is likely that by a careful analysis one can prove that we may take $k=0$ in Lemma \ref{CDHN} and that the map we construct is the inverse to the isomorphism of \textit{loc.cit.}
\end{Rem}

%%%%%%%%%%%%%%%%%%%%%%
%Automorphic L-invariants
%%%%%%%%%%%%%%%%%%%%%

\section{Automorphic L-invariants}\label{mainsection}

\subsection{Cohomology of $\p$-arithmetic groups}\label{Cohomology}
Throughout this section we fix a ring $R$.
Given a compact, open subgroup $K^\p\subseteq G(\A\pinfty)$, a smooth $G_\p$-module $M\in\ob(\mathfrak{C}^{\sm}_{R}(G_\p))$ and an $R[G(F)]$-module $N$ we define $\Ah_R(K^{\p},M;N)$ as the space of all $R$-linear maps $\Phi\colon G(\A\pinfty)/K^\p \times M\to N$.
The $R$-module $\Ah_R(K^\p,M;N)$ carries a natural $G(F)$-action given by
$$(\gamma.\Phi)(g,m)=\gamma.(\Phi(\gamma^{-1}g,\gamma^{-1}.m)).$$

\begin{Exa}
If $M$ is of the form $\cind_{K_\p}^{G_\p} R$ for some compact, open subgroup $K_\p\subseteq G_\p$, we put
$$\Ah(K^\p\times K_\p;N)=\Ah_R(K^\p,M;N).$$
There is a canonical isomorphism
$$\Ah(K^\p\times K_\p;N)\xrightarrow{\cong}C(G(\A^\infty)/(K^\p\times K_\p),N).$$

There is also the following more general version of this isomorphism:
Recall that given a group $H$, a subgroup $K\subseteq H$ and an $R[K]$-module $L$ the \emph{coinduction} $\Coind_{K}^{H}L$ of $L$ from $K$ to $H$ is the space of all functions $f\colon H\to L$ such that $f(hk)=k^{-1}f(h)$ for all $k\in k, h\in H$.
It is a $H$-representation via left multiplication.
For every $L\in\ob(\mathfrak{C}^{\sm}_{R}(G_\p))$ there is a canonical isomorphism
$$\Ah_R(K^\p,\cind_{K_\p}^{G_\p} L;R)\xrightarrow{\cong}\Coind_{K^\p\times K_\p}^{G(\A^\infty)}\Hom_R(L,R),$$
where $K^\p\times K_\p$ acts on $\Hom_R(L,R)$ via the projection onto $K_\p.$
\end{Exa}

\begin{Pro}\label{FlachundNoethersch}
Let $K^\p\subseteq G(\A\pinfty)$ be a compact, open subgroup and $M\in\ob(\mathfrak{C}^{\sm}_{R}(G_\p))$ a flawless representation.
Furthermore, let $\epsilon\colon \pi_0(G_\infty)\to\left\{\pm 1\right\}$ be a character.
\begin{enumerate}[(a)]
\item\label{FuN} The $R$-module $\HH^{d}(G(F),\Ah_R(K^\p,M;R(\epsilon)))$ is finitely generated for all $d$ if $R$ is Noetherian.
\item\label{FN2} If $N$ is a flat $R$-module (endowed with the trivial G(F)-action), then the canonical map
	 $$\HH^{d}(G(F),\Ah_R(K^\p,M;R(\epsilon)))\otimes_R N \too \HH^{d}(G(F),\Ah_R(K^\p,M;N(\epsilon)))$$
	is an isomorphism for all $d\in\Z$.
\item\label{FN3} If $R=\Omega$ is a field of characteristic $0$, then
$$\HH^{d}(G(F),\Ah_\Omega(K^\p,M;\Omega(\epsilon)))= 0$$
for all $d\gg 0$.
\end{enumerate}
\end{Pro}
\begin{proof}
It is enough to consider the case that
$$M=\cind_{K_\p}^{G_\p}L,$$
where $K_\p\subseteq G_\p$ is a compact, open subgroup and $L\in\ob(\mathfrak{C}^{\sm}_{R}(K_\p))$ finitely generated projective over $R$.
As mentioned above we have a canonical isomorphism
$$\Ah_R(K^\p,M;R(\epsilon))\xlongrightarrow{\cong}\Coind^{G(\A^{\infty})}_{K^\p\times K_\p}\Hom_R(L,R(\epsilon)).$$
By Borel's theorem on the finiteness of class numbers of algebraic groups over number fields (see Theorem 5.1 of \cite{BoCN}) the double quotient $G(F)\backslash G(\A^\infty)/(K^\p\times K_\p)$ is finite.
We fix a system of representatives $g_1,\ldots,g_n$ and put $\Gamma_{g_i}=G(F)\cap g_i (K\times K_\p) g_i^{-1}.$
Then, Shapiro’s Lemma implies that there is an isomorphism
$$\HH^{d}(G(F),\Coind^{G(\A^{\infty})}_{K^\p\times K_\p}\Hom_R(L,R(\epsilon)))\xrightarrow{\cong}\bigoplus_{i=1}^{n}\HH^{d}(\Gamma_{g_i}, \Hom_{R}(L,R(\epsilon))).$$
Now all claims follow since arithmetic groups are of type (VFL). 
\end{proof}

\begin{Rem}
The group $\pi_0(G_\infty)$ is a finite, abelian, 2-torsion group (see \cite{Mat1964}, Theorem 1).
\end{Rem}

\begin{Cor}\label{limits}
Let $(R,\m)$ be a complete DVR with finite residue field.
Under the same hypothesis as in the proposition above the canonical map
$$\HH^{d}(G(F),\Ah_R(K^\p,M;R(\epsilon)))\too \varprojlim_r \HH^{d}(G(F),\Ah_R(K^\p,M;R/\m^{r}(\epsilon)))$$
is an isomorphism for all $d$.
\end{Cor}
\begin{proof}
It suffices to proof that the projective system
$$\left(\HH^{d}(G(F),\Ah_R(K^\p,M;R/\m^{r}(\epsilon)))\right)_r$$
fulfils the Mittag-Leffler condition.
But this is true by Proposition \ref{FlachundNoethersch} \eqref{FuN}.
\end{proof}

For a compact, open subset $K^\p\times K_\p\subseteq G(\A^\infty)$ and a $R[G(F)]$-module $N$ we define
$$\HH^{d}(X_{K^\p\times K_\p},N)^{\epsilon}=\HH^{d}(G(F),\Ah(K^\p\times K_\p;N(\epsilon))).$$
Let us give a brief justification for this notation:
assume that $K^\p\times K_\p\subseteq G(\A^\infty)$ is a neat subgroup and $2$ is invertible in $R$.
Via the decomposition
$$\bigoplus_{\epsilon\colon\pi_0(G_\infty)\to \pm 1}\hspace{-1em} N(\epsilon)\cong N\otimes_R R[\pi_0(G_\infty)]$$
we see that
$$\HH^{d}(X_{K^\p\times K_\p},N)^{\epsilon}\cong\HH^{d}(G(F)^+,\Ah(K^\p\times K_\p;N))^{\epsilon},$$
where $G(F)^+$ denotes the intersection of $G(F)$ with the connected component of $G_\infty$ and the superscript $\epsilon$ on the right denotes taking $\epsilon$-isotypic component with respect to the action of $\pi_0(G_\infty)$.
By Borel's finiteness theorem the double quotient
$G(F)^+\backslash G(\A^\infty)/(K^\p\times K_\p)$ is finite.
As before, we fix a system of representatives $g_1,\ldots,g_n$ and put $\Gamma^+_{g_i}=G(F)^+\cap g_i (K\times K_\p) g_i^{-1}.$
Shapiro's Lemma implies that
$$\HH^{d}(G(F)^+,\Ah(K^\p\times K_\p;N))\cong \bigoplus_{i=1}^{h}\HH^{i}(\Gamma^+_{g_i},N).$$
On the other side let us consider the manifold
$$\mathcal{X}_{K^\p\times K_\p}=G(F)\backslash G(\A)/(K^\p \times K_\p \times K_\infty^\circ),$$
where $K_\infty^\circ\subset K_\infty$ denotes the connected component of the identity.
Since the inclusion $K_\infty\into G_\infty$ is a homotopy equivalence, we see that the space
$$\mathcal{X}=G_\infty/K_\infty$$
is contractible and that
\begin{align*}
\mathcal{X}_{K^\p\times K_\p}
\cong G(F)^+\backslash (G(\A^\infty)/(K^\p \times K_\p) \times \mathcal{X})
=\bigcup_{i=1}^{h}\Gamma_i^+ \backslash \mathcal{X}.
\end{align*}
By our neatness assumption $\Gamma_i^+$ acts properly discontinuously on $\mathcal{X}$ and thus
$$\HH^{i}(\Gamma_i^+,N)=\HH^i(\Gamma_i^+ \backslash \mathcal{X}, \underline{N}),$$
where $\underline{N}$ is the locally constant sheaf associated to $N$.
The group $\pi_0(K_\infty)\cong \pi_0(G_\infty)$ naturally acts on $\mathcal{X}_{K^\p\times K_\p}$.
We have shown that
$$\HH^{d}(X_{K^\p\times K_\p},N)^{\epsilon}\cong \HH^{d}(\mathcal{X}_{K^\p\times K_\p},\underline{N})^{\epsilon},$$
where, as before, the superscript $\epsilon$ on the right hand side denotes taking $\epsilon$-isotypic component.

Let us put
$$\tilde{\Ah}(K^\p;N(\epsilon))= \varinjlim_{K_\p}\Ah(K^\p\times K_\p;N(\epsilon)),$$
where the injective limit runs over all compact, open subgroups $K_\p \subseteq G_\p$.
Right translation defines a smooth $G_\p$-action on $\tilde{\Ah}(K;N(\epsilon))$, which commutes with the $G(F)$-action.

Utilizing that $\p$-arithmetic groups are of type $(VFL)$ similar arguments as in the proof of Proposition \ref{FlachundNoethersch} show that the canonical map
\begin{align}\label{injlimit}
\varinjlim_{K_\p}\HH^{d}(X_{K^\p\times K_\p},N)^\epsilon\xrightarrow{\cong}\HH^{d}(G(F),\tilde{\Ah}(K^\p;N(\epsilon)))
\end{align}
is an isomorphism of smooth $G_\p$-modules.

\begin{Lem}\label{hom}
Let $M$ be a smooth $R[G_\p]$-module and $N$ an $R[G(F)]$-module.
\begin{enumerate}[(a)]
\item There is a natural isomorphism
$$\HH^{0}(G(F),\Ah_R(K^\p,M;N(\epsilon)))\cong \Hom_{R[G_\p]}(M,\varinjlim_{K_\p}\HH^{0}(X_{K^\p\times K_\p},N)^\epsilon),$$
where the injective limit runs over all compact, open subgroups $K_\p \subseteq G_\p$.
\item Let $P\in\mathfrak{C}^{\sm}_{R}(G_\p)$ be a projective object.
Then there are natural isomorphisms
\begin{align}\label{homologicalalgebra}
\HH^{d}(G(F),\Ah_\Omega(K^\p,P;N(\epsilon)))\cong \Hom_{R[G_\p]}(P,\varinjlim_{K_\p}\HH^{d}(X_{K^\p\times K_\p},N)^\epsilon)
\end{align}
for all $d\geq 0$.
\end{enumerate}
\end{Lem}
\begin{proof}
We may write 
$$\tilde{\Ah}(K^\p;N(\epsilon))\cong \varinjlim_{K_\p} C(G_\p/K_\p,\Ah_R(K^\p,R;N(\epsilon))).$$
In other words, as a $G_\p$-representation $\tilde{\Ah}(K^\p;N(\epsilon))$ equals the smooth vectors of the coinduction of $\Ah_R(K^\p,R;N(\epsilon))$ from the trivial group to $G_\p$.
Therefore, by Frobenius reciprocity we have a canonical isomorphism
\begin{align*}
\Hom_{R[G_\p]}(M,\tilde{\Ah}(K^\p;N(\epsilon)))
&\cong\Hom_{R}(M,\Ah_R(K^\p,R;N(\epsilon)))\\
&\cong\Ah_R(K^\p,M;N(\epsilon)).
\end{align*}
This in turn implies that
\begin{align*}
\Hom_{R[G_\p]}(M,\HH^{0}(G(F),\tilde{\Ah}(K^\p;N(\epsilon))))
&\cong\Hom_{R[G_\p]}(M,\Hom_{R[G(F)]}(R,\tilde{\Ah}(K^\p;N(\epsilon))))\\
&\cong\Hom_{R[G(F)]}(R,\Hom_{R[G_\p]}(M,\tilde{\Ah}(K^\p;N(\epsilon))))\\
&\cong\Hom_{R[G(F)]}(R,\Ah_R(K^\p,M;N(\epsilon)))
\end{align*}
and, thus, by \eqref{injlimit} the first claim follows.

We denote by $R[G_\p]\Mod$ (respectively $R\Mod$) the category of all $R[G_p]$-modules (respectively $R$-modules).
Let $P\in\ob(\mathfrak{C}^{\sm}_{R}(G_\p))$ be a projective object.
Given an $R$-module $N$ we consider $\varinjlim_{K_\p} C(G_\p/K_\p, N)$ as a smooth $G_\p$-module via right translation.
By Frobenius reciprocity we have
$$\Hom_{R}(P,N)=\Hom_{R[G_\p]}(P, \varinjlim_{K_\p} C(G_\p/K_\p, N))$$
for any $R$-module $N$ and, therefore, $P$ is also projective as an $R$-module.
In particular, $P$ is flat as an $R$-module.
Hence, the functor
$$R[G(F)]\Mod \too R[G(F)]\Mod,\quad N \mapstoo \cind_{K^\p}^{G(\A\pinfty)}R\otimes_R P\otimes_R N(\epsilon) $$
is exact.
Thus, its right-adjoint
$$\mathcal{F}\colon R[G(F)]\Mod \too R[G(F)]\Mod,\quad N \mapstoo \Ah_R(K^\p,P;N(\epsilon))$$
sends injective objects to injective objects (see for example \cite{Weibel}, Proposition 2.3.10).
Since $P$ is projective as an $R$-module, we also see that $\mathcal{F}$ is exact.
Therefore, we see that the functors
$$R[G_\p]\Mod \too R\Mod,\quad N\mapstoo \HH^{d}(G(F),\Ah_R(K^\p,P;N(\epsilon)))$$
form an effaceable $\delta$-functor and, thus, agree with the right derived functor of
$$R[G_\p]\Mod \too R\Mod,\quad N\mapstoo \HH^{0}(G(F),\Ah_R(K^\p,P;N(\epsilon))).$$

Conversely, since $P\in\ob(\mathfrak{C}^{\sm}_{R}(G_\p))$ is projective, the functors
$$R[G_\p]\Mod\too R\Mod,\quad N\mapstoo \Hom_{R[G_\p]}(P,\varinjlim_{K_\p}\HH^{d}(X_{K^\p\times K_\p},N)^\epsilon)$$
form a $\delta$-functor.
By similar arguments as above we see that this $\delta$-functor is effaceable and, therefore, agrees with the right derived functor of
$$R[G_\p]\Mod\too R\Mod,\quad N \mapstoo \Hom_{R[G_\p]}(P,\varinjlim_{K_\p}\HH^{0}(X_{K^\p\times K_\p},N)^\epsilon).$$
Thus the second claim follows from the first.
\end{proof}

\subsection{The $\pi$-isotypic component}\label{Component}
We determine the $\pi$-isotypic component of various cohomology groups.

Since $v_{I}^{\infty}(R)=v_I^{\infty}(\Z)\otimes R$ for all rings $R$, we have a canonical isomorphism
$$\Ah_\Z(K^\p,v^{\infty}_{I}(\Z);N)\cong \Ah_R(K^\p,v^{\infty}_{I}(R);N)$$
for any $R$-module $N$.
Hence, we abbreviate this space by $\Ah(K^\p,v^{\infty}_{I};N)$ (and similarly for $\St^{\infty}_{G_\p}$ in place of $v^{\infty}_I$).

By \cite{FabianRat}, Theorem C, there exists a number field $\Q_\pi\subseteq \C$ such that $\pi^\infty$ has a model over $\Q_\pi$ that is essentially unique.
We fix a field extension $\Omega$ of $\Q_\pi$ and a compact, open subgroup $K^\p\subseteq G(\A\pinfty)$ such that $(\pi\pinfty)^{K^\p}\neq 0$.
Let
$$\mathbb{\T}=\mathbb{T}(K^\p)_\Omega=C_c(K^\p\backslash G(\A\pinfty)/K^\p,\Omega)$$
be the $\Omega$-valued Hecke algebra of level $K^\p$ away from $\p$.
By abuse of notation we denote the model of $\pi^\infty$ over $\Omega$ induced by the one over $\Q_\pi$ also by $\pi^\infty$.
If $V$ is a $\mathbb{T}(K^\p)_\Omega$-module, we put
$$V[\pi]=\Hom_{\mathbb{T}}((\pi\pinfty)^{K^\p},V).$$

Let $I_\p\subseteq G_\p$ be an Iwahori subgroup.
It is well-known (see for example the discussion above Corollary 2 of \cite{GK2}) that
\begin{align}\label{unIw}\dim_\Omega \St_{G_\p}(\Omega)^{I_\p}=1.\end{align}
We fix a non-zero vector $\varphi\in \St_{G_\p}(\Omega)^{I_\p}.$
By Frobenius reciprocity the vector $\varphi$ gives rise to a $G_\p$-equivariant homomorphism
$$\cind_{I_\p}^{G_\p}\Omega \too \St_{G_\p}(\Omega)$$
that is surjective by the irreducibility of $\St_{G_\p}(\Omega)$.
We denote the induced Hecke-equivariant map in cohomology by
\begin{align}\label{evaluation}
\ev^{(d)}\colon \HH^{d}(G(F),\Ah(K^\p,\St^{\infty}_{G_\p};\Omega(\epsilon)))\too \HH^{d}(X_{K^\p\times I_\p},\Omega)^{\epsilon}.
\end{align}
Note that \eqref{unIw} implies that $\ev^{(d)}$ only depends on $\varphi$ up to multiplication by a non-zero scalar.

\begin{Pro}\label{dimensions}
The following holds:
\begin{enumerate}[(a)]
\item\label{firstdim} For every character $\epsilon\colon \pi_0(G_\infty)\to\left\{\pm 1 \right\}$ we have
 $$\dim_\Omega \HH^{d}(X_{K^\p\times I_\p},\Omega)^{\epsilon}[\pi]= m_\pi \cdot\binom{\delta}{d-q}.$$
\item\label{seconddim} The map $\ev^{(d)}$ induces an isomorphism
$$\HH^{d}(G(F),\Ah(K^\p,\St^{\infty}_{G_\p};\Omega(\epsilon)))[\pi]\xrightarrow{\ev^{(d)}} \HH^{d}(X_{K^\p\times I_\p},\Omega)^{\epsilon}[\pi]$$
for all $d$.
\item\label{thirddim} For every character $\epsilon\colon \pi_0(G_\infty)\to\left\{\pm 1 \right\}$ we have
$$\dim_\Omega \HH^{d}(G(F),\Ah(K^\p,\St^{\infty}_{G_\p};\Omega(\epsilon)))[\pi] = m_\pi \cdot\binom{\delta}{d-q}.$$
\end{enumerate}
\end{Pro}
\begin{proof}
The last claim is a direct consequence of the first two.
It is enough to proof the first claim in the case $\Omega=\C$ by Proposition \ref{FlachundNoethersch} \eqref{FN2}.
But, in that case the claim follows from the computation of $(\mathfrak{g},K_\infty^{\circ})$-cohomology of tempered representations (see Theorem III.5.1 of \cite{BW} and also 5.5 of \cite{Bo} for the non-compact case) and our strong multiplicity one hypothesis on $\pi$.

If $K_\p\subseteq G_\p$ is a parahoric subgroup which is not Iwahori, the strong multiplicity one hypothesis implies that
$$\HH^{i}(X_{K^\p\times K_\p},\Omega)^{\epsilon}[\pi]=0$$
for all $i$.
Thus, the second assertion can be deduced from the resolution of the Steinberg representation given in Theorem \ref{resolution1}.
\end{proof}

Let $J\subseteq I\subseteq \Delta$ be two subset with $|I|=|J|+1$.
The non-split short exact sequence
$$0\too v^{\infty}_J(\Q)\too \mathcal{E}_{I,J}^{\infty}(\Q) \too v^{\infty}_I(\Q)\too 0$$
induces a short exact sequence
$$0\to \Ah(K^\p,v^{\infty}_I;\Omega(\epsilon)) \to \Ah_\Q(K,\mathcal{E}_{I,J}^{\infty}(\Q);\Omega(\epsilon))\to \Ah(K^\p,v^{\infty}_J;\Omega(\epsilon)) \to 0.$$
The induced boundary map in cohomology is Hecke-equivariant and, therefore, yields a map
\begin{align*}
c^{(d)}_{I,J}[\pi]^{\epsilon}\colon\HH^{d}(G(F),\Ah(K^\p,v^{\infty}_J;\Omega(\epsilon)))[\pi]
\to\HH^{d+1}(G(F),\Ah(K^\p,v^{\infty}_I;\Omega(\epsilon)))[\pi].
\end{align*}
(Note that the map is only well-defined up to a non-zero rational scalar.)

\begin{Lem}\label{Vytas}Let $V\in\mathfrak{C}^{\sm}_{\Omega}(G_\p)$ be an admissible smooth representation of finite length such that
$$\Ext_{\sm}^{d}(V,\St^{\infty}_{G_\p}(\Omega))=0\quad \forall d\geq 0.$$
Then:
$$\HH^{d}(G(F),\Ah(K^\p,V;\Omega(\epsilon)))[\pi]=0\quad \forall d\geq 0.$$
\end{Lem}
\begin{proof}
By Theorem \ref{admissible} the representation $V$ is flawless.
We fix a resolution $$0\too P_m\too\cdots\too P_0\too V \too 0.$$
The spectral sequence for double complexes gives a spectral sequence starting at
$$E_1^{i,j}= \HH^{j}(G(F),\Ah(K^\p,P_{-i};\Omega(\epsilon)))$$
and converging to
$$E^{d}_{\infty}=\HH^{d}(G(F),\Ah(K^\p,V;\Omega(\epsilon))).$$
By Hypothesis \ref{Hyp} all $[\pi]$-isotypic subquotients of cohomology lie in the cuspidal part.
Since the space of cuspidal automorphic representations is completely reducible by the classical theorem of Gelfand and Piatetski-Shapiro (see \cite{GPS}), we may pass to $[\pi]$-isotypic components.
We want to show that already on the second sheet of the spectral sequence all terms vanish.

By Lemma \ref{hom} we have natural isomorphisms
\begin{align*}
\HH^{j}(G(F),\Ah(K^\p,P;\Omega(\epsilon)))[\pi]
\cong \Hom_{G_\p}(P,\varinjlim_{K_\p}\HH^{j}(X_{K^\p \times K_\p},\Omega)^{\epsilon})[\pi]).
\end{align*}
for any projective smooth representation $P\in\mathfrak{C}^{\sm}_{\Omega}(G_\p)$.
By our strong multiplicity one hypothesis we know that
$$\varinjlim_{K_\p}\HH^{j}(X_{K^\p\times K_\p},\Omega)^{\epsilon}[\pi]\cong \St^{\infty}_{G_\p}(\Omega)^{m_\pi\cdot\binom{\delta}{j-q}}$$
and, therefore, we have natural isomorphisms
$$\HH^{j}(G(F),\Ah(K^\p,P;\Omega(\epsilon)))[\pi]\cong \Hom_{G_\p}(P,\St^{\infty}_{G_\p}(\Omega))^{m_\pi\cdot\binom{\delta}{j-q}}.$$
Since the representations $P_j$ are projective objects in $\mathfrak{C}^{\sm}_{\Omega}(G_\p)$ we see that
$$E_2^{i,j}[\pi]\cong \Ext_{\sm}^{i}(V,\St_{G_\p}^{\infty}(\Omega))^{m_\pi\cdot\binom{\delta}{j-q}}=0$$
holds by our assumption on $V$.
\end{proof}

As a corollary of the lemma above and Corollary \ref{extvanishing} we get:
\begin{Cor}\label{smoothcup}
The map $c^{(d)}_{I,J}[\pi]^{\epsilon}$ is an isomorphism for all subsets $J\subseteq I\subseteq \Delta$ with $|I|=|J|+1$ and every character $\epsilon\colon \pi_0(G_\infty)\to\left\{\pm 1 \right\}$.
\end{Cor}

This together with Proposition \ref{dimensions} \eqref{thirddim} implies:
\begin{Pro}\label{fourthdim}
For every character $\epsilon\colon \pi_0(G_\infty)\to\left\{\pm 1 \right\}$ and every subset $I\subseteq \Delta$
we have
$$\dim_\Omega \HH^{d+|I|}(G(F),\Ah(K^\p,v^{\infty}_I;\Omega(\epsilon)))[\pi] = m_\pi \cdot\binom{\delta}{d-q}.$$
\end{Pro}

\begin{Rem}\begin{enumerate}\thmenumhspace
\item The higher smooth Ext-groups of the Steinberg representation with itself vanish by the theorem of Dat and Orlik.
Hence, a spectral sequence argument similar to the one above yields an alternative proof of Proposition \ref{dimensions} \eqref{seconddim}.
\item All arguments of this and the previous section carry over to cohomology with values in an arbitrary finite-dimensional algebraic coefficient system.
\item Corollary \ref{smoothcup} was proven in the case $G=PGL_2$ by Spieß using an explicit resolution of the representation $\mathcal{E}_{I,J}^{\infty}(\Q)$ (see \cite{Sp}, Lemma 6.2).
In \cite{AdS} Alon and de Shalit prove a version of the above corollary for the cohomology of cocompact, discrete subgroups of $PGL_d(F_\p)$.
See also the article \cite{GKomega} of Gro\ss e-Klönne for the case of non-trivial coefficient systems.
\end{enumerate}
\end{Rem}

\subsection{Automorphic $\LI$-invariants}\label{Definition} 
Let $\Omega$ be a finite extension of $\Q_\p$, which contains $\Q_\pi$, with ring of integers $R$.
If $V$ is a continuous representation of $G_\p$ on a locally convex $\Omega$-vector space, we write $V^{\vee}=\Hom_{\cont}(V,\Omega)$ for its continuous dual.
Further, we write $\Ah_{\Omega}^{\cont}(K^{\p},V;\Omega(\epsilon))$ for the $\Omega[G(F)]$-module of functions from $G(\A\pinfty)/K^\p$ to $V^{\vee}(\epsilon)$.
We regularly drop the subscript $\Omega$, when it is clear from the context.

\begin{Pro}\label{automaticuniversal}
Suppose that $V$ is a smooth $\Omega$-representation of $G_\p$ that admits a flawless $G_\p$-stable $R$-lattice.
Then $V$ admits a universal unitary completion $V^{\univ}$ and the canonical map
$$\HH^{d}(G(F),\Ah^{\cont}_{\Omega}(K^\p,V^{\univ};\Omega(\epsilon)))\too \HH^{d}(G(F),\Ah_{\Omega}(K^\p,V;\Omega(\epsilon)))$$
is an isomorphism for every character $\epsilon\colon \pi_0(G_\infty)\to\left\{\pm 1\right\}$ and every compact, open subgroup $K^\p\subseteq G(\A\pinfty)$.
\end{Pro}
\begin{proof}
Let $M$ be a flawless $R$-lattice in $V$.
Since $M$ is finitely generated as an $R[G_\p]$-module, the completion of $V$ with respect to $M$ is the universal unitary completion $V^{\univ}$ of $V$ by \cite{EmUnitary}, Proposition 1.17.
In particular, we have an isomorphism
$$\Hom_R(M,R)\otimes_R\Omega\cong\Hom_{\cont}(V^{\univ},\Omega).$$
The claim now follows from Proposition \ref{FlachundNoethersch} \eqref{FN2}.
\end{proof}

As $\St_{G_\p}^{\cont}(\Omega)$ is the universal unitary completion of $\St_{G_\p}^{\infty}(\Omega)$ by Corollary \ref{universal}, Proposition \ref{automaticuniversal} yields an isomorphism
\begin{align*}
\HH^{d}(G(F),\Ah(K^\p,\St_{G_\p}^{\infty};\Omega(\epsilon)))\xlongrightarrow{\cong} \HH^{d}(G(F),\Ah^{\cont}(K,\St_{G_\p}^{\cont};\Omega(\epsilon)))
\end{align*}
for every $d\geq 0$.
Thus, by taking pullback via the embedding $\St_{G_\p}^{\an}(\Omega)\subseteq \St_{G_\p}^{\cont}(\Omega)$ we get a map
\begin{align*}
\HH^{d}(G(F),\Ah(K^\p,\St_{G_\p}^{\infty};\Omega(\epsilon)))\too \HH^{d}(G(F),\Ah^{\cont}(K,\St_{G_\p}^{\an};\Omega(\epsilon))).
\end{align*}
Therefore, for every $i\in\Delta$ we get a well-defined cup-product pairing
\begin{align*}
&\HH^{d}(G(F),\Ah(K^\p,\St_{G_\p}^{\infty};\Omega(\epsilon))) \times \Ext^{1}_{\an}(v_{i}^{\infty},\St_{G_\p}^{\an})\\
\too &\HH^{d+1}(G(F),\Ah(K^\p,v^{\infty}_{i};\Omega(\epsilon)))
\end{align*}
which commutes with the Hecke-action.
By Theorem \ref{Ding} we have a canonical isomorphism $\Hom_{\cont}(F_\p^{\times},\Omega)\cong\Ext^{1}_{\an}(v^{\infty}_{i},\St_{G_\p}^{\an})$.

Hence, taking cup product with the extension $\mathcal{E}^{\an}_{i,\emptyset}(\lambda\circ i)$ associated to a homomorphism $\lambda\in \Hom_{\cont}(F_\p^{\times},\Omega)$ yields a map
$$c^{(d)}_{i}(\lambda)[\pi]^{\epsilon}\colon \HH^{d}(G(F),\Ah(K^\p,\St_{G_\p}^{\infty};\Omega(\epsilon)))[\pi]\too \HH^{d+1}(G(F),\Ah(K^\p,v^{\infty}_{i};\Omega(\epsilon)))[\pi]$$
on $\pi$-isotypic parts.

\begin{Def}\label{deficom}
Given a character $\epsilon\colon \pi_0(G_\infty) \to \left\{\pm 1\right\}$, an integer $d\in \Z$ with $0\leq d\leq \delta$ and a root $i\in\Delta$ we define
$$\LI_{i}^{(d)}(\pi,\p)^{\epsilon}\subseteq \Hom_{\cont}(F_\p^{\times},\Omega)$$
as the kernel of the map $\lambda \mapsto c^{(d+q)}_{i}(\lambda)[\pi]^{\epsilon}$.
\end{Def}

\begin{Rem}\label{flawlessversion}
Suppose that $v^{\infty}_{i}(R)$ is flawless.
With similar arguments as above, we would have a canonical isomorphism
$$\HH^{d}(G(F),\Ah(K^\p,v_{i}^{\infty};\Omega(\epsilon)))\too \HH^{d}(G(F),\Ah_\Omega^{\cont}(K,\VV^{\cont}_{i};\Omega(\epsilon)))$$
of cohomology groups for every root $i\in\Delta$.
In this case, $\LI$-invariants defined by taking the cup product with the continuous extensions classes $\mathcal{E}^{\cont}_{i,\emptyset}(\lambda\circ i)$ would yield the same result as their locally analytic counterparts.
\end{Rem}

\begin{Pro}\label{codimension}
For all sign characters $\epsilon$, every degree $d\in [0,\delta]\cap \Z$ and every root $i \in \Delta$ the $\LI$-invariant
$$\LI_{i}^{(d)}(\pi,\p)^{\epsilon}\subseteq \Hom_{\cont}(F_\p^{\times},\Omega)$$
is a subspace of codimension at least one, which does not contain the space of smooth extensions.

Suppose $m_\pi=1$.
Then, in the cases $d=0$ and $d=\delta$ the codimension is exactly one.
\end{Pro}
\begin{proof}
The first assertion follows directly from Corollary \ref{smoothcup}.
In the case $i=0$ or $d=\delta$ the cohomology groups in question are one-dimensional by Proposition \ref{fourthdim} and therefore, the second assertion follows.
\end{proof}

\begin{Con}\label{conjecture}
Let $i\in \Delta$ be a root.
\begin{enumerate}[(i)]
\item $\LI_{i}^{(d)}(\pi,\p)^{\epsilon}\subseteq\Hom_{\cont}(F_\p^{\times},\Omega)$ has codimension one for all $0\leq d\leq \delta$ and each sign character $\epsilon$.
\item\label{conjdegree} $\LI_{i}^{(d)}(\pi,\p)^{\epsilon}$ does not depend on the degree $d\in [0,\delta]\cap \Z$.
\item\label{conjchar} $\LI_{i}^{(d)}(\pi,\p)^{\epsilon}$ does not depend on the sign character $\epsilon$.
\end{enumerate}
\end{Con}

\begin{Rem}\label{independenceRem}
Let us recall the status of these conjectures for the group $G=PGL_2$.
In the case $F=\Q$ it was shown by Darmon that the $\LI$-invariant does not depend on the sign character $\epsilon$ (see Section 3 of \cite{D}).
Alternative proofs were given by Bertolini--Darmon--Iovita (see Theorem 6.8 of \cite{BDI}) and Breuil (see Corollary 5.1.3 of \cite{Br2}).
For arbitrary number fields only partial results are known (see Theorem A of \cite{Ge2}).

In \cite{Ge3} it is shown that Venkatesh's conjectures (cf.~\cite{Ve}) on the action of the derived Hecke algebra on cohomology imply that automorphic $\LI$-invariants do not depend on the cohomological degree $d$.
This result is extended to more general groups in Section \ref{DHa}.
\end{Rem}

\subsection{Invariance under twisting and Galois actions}\label{properties}
We use the same notations as in the previous section.
We fix a character $\epsilon\colon \pi_0(G_\infty) \to \left\{\pm 1\right\}$, a simple root $i\in \Delta$ and an integer $d$ between $0$ and $\delta$.

Let $\chi\colon G(F)\backslash G(\A)\to \C^\times$ be a locally constant character, whose local component at $\p$ is trivial.
The automorphic representation $\pi\otimes\chi$ fulfils all the assumptions we imposed on $\pi$.
Thus after enlarging $\Omega$ if necessary, we can define
$$\LI_{i}^{(d)}(\pi\otimes\chi,\p)^{\epsilon}\subseteq \Hom_{\cont}(F_\p^{\times},\Omega).$$
The restriction of $\chi$ to $G(F_\infty)$ descends to a character $\chi_\infty\colon \pi_0(G_\infty)\to \left\{\pm 1\right\}$.

\begin{Pro}
We have
$$\LI_{i}^{(d)}(\pi\otimes\chi,\p)^{\epsilon\chi_\infty}=\LI_{i}^{(d)}(\pi,\p)^{\epsilon}$$
for every locally constant character $\chi\colon G(F)\backslash G(\A)\to \C^\times$ with $\chi_\p=1$.
\end{Pro}
\begin{proof}
After possibly shrinking the compact, open subgroup $K^\p\subseteq G(\A\pinfty)$ chosen at the beginning of Section \ref{Definition} we may assume that the restriction of $\chi$ to it is trivial.
Let $\mathfrak{Tw}_\chi\colon \Ah(K^\p,v^{\infty}_{I};\Omega(\epsilon))\to \Ah(K^\p,v^{\infty}_{I};\Omega(\epsilon\chi_\infty))$ be the $G(F)$-equivariant map given by $\mathfrak{Tw}_\chi(\phi)(g,m)=\chi(g)\cdot \phi(g,m).$
The induced map in cohomology 
$$\HH^d(G(F),\Ah(K^\p,v^{\infty}_{I};\Omega(\epsilon)))[\pi]\xrightarrow{\mathfrak{Tw}_\chi}\HH^d(G(F),\Ah(K^\p,v^{\infty}_{I};\Omega(\epsilon\chi_\infty)))[\pi\otimes\chi]$$
is an isomorphism for all subsets $I\subseteq \Delta$.
The claim now follows from the commutativity of the diagram
\begin{center}
\begin{tikzpicture}
    \path 	(0,0) 	node[name=A]{$\HH^d(G(F),\Ah(K^\p,\St_{G_\p};\Omega(\epsilon)))[\pi]$}
		(6.6,0) 	node[name=B]{$\HH^d(G(F),\Ah(K^\p,\St_{G_\p};\Omega(\epsilon\chi_\infty)))[\pi\otimes\chi]$}
		(0,-1.5) 	node[name=C]{$\HH^{d+1}(G(F),\Ah(K^\p,v^{\infty}_{i};\Omega(\epsilon)))[\pi]$}
		(6.6,-1.5) 	node[name=D]{$\HH^{d+1}(G(F),\Ah(K^\p,v^{\infty}_{i};\Omega(\epsilon\chi_\infty)))[\pi\otimes\chi]$};
    \draw[->] (A) -- (C) node[midway, left]{$c^{(d)}_{i}(\lambda)[\pi]^{\epsilon}$};
    \draw[->] (A) -- (B) node[midway, above]{$\mathfrak{Tw}_\chi$};
    \draw[->] (C) -- (D) node[midway, above]{$\mathfrak{Tw}_\chi$};
    \draw[->] (B) -- (D) node[midway, right]{$c^{(d)}_{i}(\lambda)[\pi\otimes\chi]^{\epsilon\chi_\infty}$};
\end{tikzpicture} 
\end{center}
for all $\lambda\in\Hom_{\cont}(F_\p^{\times},\Omega).$
\end{proof}

Now suppose that there exists a subfield $F_0\subseteq F$ and an algebraic group $G_0$ over $F_0$ such that
\begin{itemize}
\item $F$ is a Galois extension of $F_0$,
\item the base change of $G_0$ to $F$ is isomorphic to $G$ and
\item $G_0$ is split at the prime lying below $\p$.
\end{itemize}

The $F_0$-linear action of the Galois group $\Gal(F/F_0)$ on $\A$ induces an action on $G(\A)=G_0(\A)$.
For an element $\sigma\in \Gal(F/F_0)$ let $\pi^\sigma$ be the representation of $G(\A)$ with the same underlying vector space as $\pi$ and $G(\A)$-action given by $g\cdot_{\pi^{\sigma}}v=\sigma(g).v$ for all $g\in G(\A)$, $v\in \pi$.
The representation $\pi^\sigma$ fulfils all the assumptions we imposed on $\pi$ but with the prime $\p^{\sigma}=\sigma^{-1}(\p)$ in place of $\p$.
Thus after enlarging $\Omega$ if necessary, we can define
$$\LI_{i}^{(d)}(\pi^{\sigma},\p^{\sigma})^\epsilon\subseteq \Hom_{\cont}(F_\p^{\times},\Omega).$$
Given a character $\epsilon\colon \pi_0(G_\infty)\to \left\{\pm 1\right\}$ we put
$$\epsilon^{\sigma}\colon \pi_0(G_\infty)\xrightarrow{\sigma}\pi_0(G_\infty)\xrightarrow{\epsilon}\left\{\pm 1\right\}.$$
Pullback via $\sigma\colon F_{\p^{\sigma}}\to F_\p$ yields an isomorphism
$$\sigma^\ast\colon \Hom_{\cont}(F_\p^{\times},\Omega)\too \Hom_{\cont}(F_{\p^{\sigma}}^{\times},\Omega).$$

The following proposition follows easily by unravelling the definitions. 
See \cite{Ge2}, Lemma 3.1, for a complete proof in the $PGL_2$-case.
\begin{Pro}
We have
$$\LI_{i}^{(d)}(\pi^\sigma,\p^{\sigma})^{\epsilon^\sigma}=\sigma^\ast(\LI_{i}^{(d)}(\pi,\p)^{\epsilon})$$
for every $\sigma\in \Gal(F/F_0)$.
\end{Pro}

This together with the conjectural comparison between automorphic and Galois theoretic $\LI$-invariants, which we will explain later in Section \ref{Galois}, suggest that automorphic $\LI$-invariants should be invariant under automorphic base change:
Let us spell this out more precisely in the case $G=PGL_{n,F}$.
%To simplify the situation, we assume that $\pi^{\infty}$ can be defined over the rationals.
Let $E$ be a solvable extension of $F$.
We fix a prime $\q$ of $E$ lying above $\p$.
By the work of Arthur and Clozel (see \cite{AC}) there exists the base change lift $\pi_E$ of $\pi$ to $PGL_{n,E}$.
The representation $\pi_E$ fulfils all the assumptions we imposed on $\pi$ and, in addition, it is stable under the action of the Galois group $\Gal(E/F)$.
Thus, we get the equality
\begin{align}\label{galoisequal}
\LI_{i}^{(d)}(\pi_{E}^\sigma,\q^{\sigma})^{\cf}=\LI_{i}^{(d)}(\pi_E,\q)^{\cf}
\end{align}
for all $\sigma\in \Gal(E/F)$.
In particular, the $\LI$-invariant $\LI_{i}^{(0)}(\pi_E,\q)^{\cf}$ is the preimage of a subspace of $\Hom_{\cont}(F_\p^\times,\Omega)$ under the canonical projection
$$\pr\colon \Hom_{\cont}(E_\q^\times,\Omega) \too \Hom_{\cont}(F_\p^\times,\Omega).$$
\begin{Con}\label{galoisconj}
Under the above assumptions the equality
$$\LI_{i}^{(0)}(\pi_E,\q)^{\cf}=\pr^{-1}(\LI_{i}^{(0)}(\pi,\p)^{\cf})$$
holds in $\Hom_{\cont}(E_\q^\times,\Omega)$.
\end{Con}

Conjecture \ref{galoisconj} is proven in \cite{Ge2}, Corollary 3.7, in the case $G=PGL_2$, $F=\Q$ and $p$ totally inert in $E$ under a standard assumption on the non-vanishing of automorphic $L$-functions.
If, in addition, $E$ is an imaginary quadratic extension, Barrera and Williams have proven a higher weight analogue (see Proposition 10.02 of \cite{BWi}).

\subsection{Automorphic $\LI$-invariants relative to a set of roots}\label{relative}
We describe a variant of the definition of automorphic $\LI$-invariants, which seems natural in view of the calculation of finite polynomial cohomology of Drinfeld's upper half space in \cite{BdS}.
As before let $\Omega$ be a finite extension of $\Q_\p$, which contains $\Q_\pi$.
Let $R$ be its ring of integers.
We fix subsets $J\subseteq I\subseteq \Delta$ with $|I|=|J|+1$.
We assume throughout this section that $v^{\infty}_S(R)$ is flawless for every subset $S\subseteq \Delta$.
By Theorem $\ref{genflaw}$ this assumption holds if all simple factors of $G_{F_\p}$ are of type $A_n$.

As before, there are canonical isomorphisms
\begin{align}\label{relativeisom}
\HH^{d}(G(F),\Ah(K^\p,v^{\infty}_{J};\Omega(\epsilon)))\xlongrightarrow{\cong} \HH^{d}(G(F),\Ah^{\cont}(K^\p,\VV^{\cont}_{J};\Omega(\epsilon)))
\end{align}
by Corollary \ref{universal} and Proposition \ref{automaticuniversal}.
Therefore, taking cup product with the class $\mathcal{E}^{\an}_{I,J}(\lambda\circ i)\in \Ext^{1}_{\an}(v_I^{\infty},v_J^{\an})$ associated to a character $\lambda\in \Hom_{\cont}(F_\p^{\times},\Omega)$ induces a map
$$c^{(d)}_{I,J}(\lambda)[\pi]^{\epsilon}\colon \HH^{d}(G(F),\Ah(K^\p,v^{\infty}_{J};\Omega(\epsilon)))[\pi]\too \HH^{d+1}(G(F),\Ah(K^\p,v^{\infty}_{I};\Omega(\epsilon)))[\pi]$$
on the $\pi$-isotypic component of cohomology.

\begin{Def}
For a character $\epsilon\colon \pi_0(G_\infty) \to \left\{\pm 1\right\}$ and an integer $d\in \Z$ with $0\leq d\leq \delta$ we define
$$\LI_{I,J}^{(d)}(\pi,\p)^{\epsilon}\subseteq \Hom_{\cont}(F_\p^{\times},\Omega)$$
as the kernel of the map $\lambda \mapsto c^{(d+q+|J|)}_{I,J}(\lambda)[\pi]^{\epsilon}$.
\end{Def}

\begin{Thm}
Let $i\in I$ be the unique root which is not contained in $J$.
The equality of $\LI$-invariants
$$\LI_{I,J}^{(d)}(\pi,\p)^{\epsilon}=\LI_{i}^{(d)}(\pi,\p)^{\epsilon}$$
holds for every characters $\epsilon\colon \pi_0(G_\infty) \to \left\{\pm 1\right\}$ and all integers $d\in \Z$ with $0\leq d\leq \delta$.
\end{Thm}
\begin{proof}
For a continuous homomorphism $\lambda \colon F_\p^\times\to \Omega$ taking cup product with the class $\mathcal{E}^{\cont}_{I,J}(\lambda\circ i)$ induces the map
$$\hat{c}^{(d)}_{I,J}(\lambda)\colon \HH^{d}(G(F),\Ah(K^\p,\VV^{\cont}_{J};\Omega(\epsilon)))[\pi]\too \HH^{d+1}(G(F),\Ah(K^\p,\VV^{\cont}_{I};\Omega(\epsilon)))[\pi].$$
By construction we have the following commutative diagram
\vspace{-0.5em}
\begin{center}
\begin{tikzpicture}
    \path 	(0,0) 	node[name=A]{$\HH^{d}(G(F),\Ah^{\cont}(K^\p,\VV^{\cont}_{J};\Omega(\epsilon)))[\pi]$}
		(7.2,0) 	node[name=B]{$\HH^{d+1}(G(F),\Ah^{\cont}(K^\p,\VV^{\cont}_{I};\Omega(\epsilon)))[\pi]$}
		(0,-1.5) 	node[name=C]{$\HH^{d}(G(F),\Ah(K^\p,v^{\infty}_{J};\Omega(\epsilon)))[\pi]$}
		(7.2,-1.5) 	node[name=D]{$\HH^{d+1}(G(F),\Ah(K^\p,v^{\infty}_{I};\Omega(\epsilon)))[\pi]$};
    \draw[->] (A) -- (C) node[midway, left]{\eqref{relativeisom}} node[midway, right]{$\cong$};
    \draw[->] (A) -- (B) node[midway, above]{$\hat{c}^{(d)}_{I,J}(\lambda) $};
    \draw[->] (C) -- (D) node[midway, above]{$c^{(d)}_{I,J}(\lambda)[\pi]^{\epsilon}$};
    \draw[->] (B) -- (D) node[midway, right]{\eqref{relativeisom}}node[midway, left]{$\cong$};
\end{tikzpicture} 
\end{center}
where the vertical maps are isomorphism by our flawlessness assumption.
In particular, the map $\hat{c}^{(d)}_{I,J}(\ord_\p)$ is an isomorphism by Corollary \ref{smoothcup}.
Moreover, we have
\begin{align}\label{def2eq}
\LI_{I,J}^{(d)}(\pi,\p)^{\epsilon}=\Ker(\lambda\mapstoo \hat{c}^{(d+q+|J|)}_{I,J}(\lambda)).
\end{align}

Let $K\subseteq \Delta$ be another subset such that $J\subseteq K$, $|K|=|J|+1$ and $K\neq I$.
The diagram
\vspace{-2em}
\begin{center}
\begin{tikzpicture}
    \path 	(0,0) 	node[name=A]{$\HH^{d}(G(F),\Ah^{\cont}(K^\p,\VV^{\cont}_{J};\Omega(\epsilon)))[\pi]$}
		(7.2,0) 	node[name=B]{$\HH^{d+1}(G(F),\Ah^{\cont}(K^\p,\VV^{\cont}_{I};\Omega(\epsilon)))[\pi]$}
		(0,-1.5) 	node[name=C]{$\HH^{d+1}(G(F),\Ah^{\cont}(K^\p,\VV^{\cont}_{K};\Omega(\epsilon)))[\pi]$}
		(7.2,-1.5) 	node[name=D]{$\HH^{d+2}(G(F),\Ah(K^\p,v^{\infty}_{K\cup I};\Omega(\epsilon)))[\pi]$};
    \draw[->] (A) -- (C) node[midway, left]{$\hat{c}^{(d)}_{K,J}(\ord_\p)$} node[midway, right]{$\cong$};
    \draw[->] (A) -- (B) node[midway, above]{$\hat{c}^{(d)}_{I,J}(\lambda) $};
    \draw[->] (C) -- (D) node[midway, above]{$\hat{c}^{(d+1)}_{K \cup I,K}(\lambda)$};
    \draw[->] (B) -- (D) node[midway, right]{$\pm \hat{c}^{(d+1)}_{K\cup I,I}(\ord_\p)$}node[midway, left]{$\cong$};
\end{tikzpicture} 
\end{center}
is commutative by Lemma \ref{indext} and the vertical arrows are isomorphism.
Thus \eqref{def2eq} implies that
$$\LI_{I,J}^{(d)}(\pi,\p)^{\epsilon}=\LI_{K\cup I,K}^{(d)}(\pi,\p)^{\epsilon}.$$
The claim now follows by induction on the size of $J$.
\end{proof}

\begin{Rem}
An interesting problem is to relate $\LI_{I,J}^{(d)}(\pi,\p)^{\epsilon}$ with the transcendental $\LI$-invariants of Besser and de Shalit (cf.~Section 3 of \cite{BdS}), whenever both type of $\LI$-invariants are defined.
The author hopes to come back to this question in the future.
\end{Rem}

\subsection{Cohomology with compact support}\label{Compact}
There is also a variant of the above constructions using cohomology with compact support.
We use its description in terms of the Steinberg module of $G$.
For more details on the Steinberg module of a reductive group see \cite{R}. 
We keep the same notations as in the previous sections.

Let $\mathcal{P}$ be the set of proper maximal $F$-rational parabolic subgroups of $G$.
A subset $S=\left\{P_{0},\ldots,P_{k}\right\}\subset \mathcal{P}$ of cardinality $k+1$ is called $k$-simplex if $P_{1}\cap\ldots\cap P_{k}$ is a parabolic subgroup of $G$.
Let $D_k$ be the free abelian group generated by the $k$-simplices on $\mathcal{P}$.
Taking the associated simplicial complex we get a sequence of $G(F)$-modules
\begin{align}\label{dualizingcomplex}
D_{l-1} \too  D_{l-2}\too\cdots\too D_0\too\Z\too 0.
\end{align}
\begin{Def}
The Steinberg (or dualizing) module $D_G$ of $G$ is the kernel of the map $D_{l-1} \to  D_{l-2}$ (where we set $D_{-1}=\Z$ if $l=1$).
\end{Def} 
Let $\mathcal{P}_{k}$ be the set of proper $F$-rational parabolic subgroups of semi-simple $F$-rank $l-1-k$ containing a fixed minimal parabolic subgroup of $G$.
Then for $0\leq k\leq l-1$ there is a natural isomorphism of $G(F)$-modules
\begin{align*}
\bigoplus_{P\in\mathcal{P}_{k}}\cind_{P(F)}^{G(F)}\Z\xrightarrow{\cong} D_{k}.
\end{align*}
The homology of the complex \eqref{dualizingcomplex} can be identified with the reduced homology of the spherical building associated to $G$, which is homotopy equivalent to a bouquet of $(l-1)$-spheres.
Therefore, the complex of $G(F)$-modules
\begin{align}\label{dualizingresolution}
0\too D_{G} \too D_{l-1} \too\cdots\too D_{0}\too\Z\too 0
\end{align}
is exact (see for example \cite{BS2}).

With the same notations as at the beginning of Section \ref{Cohomology} we put
\begin{align*}
\Ah_{R,c}(K^\p,M;N)&=\Hom_\Z(D_G,\Ah_{R}(K^\p,M;N)),\\
\Ah_c(K^\p\times K_\p;N)&=\Hom_\Z(D_G,\Ah(K^\p\times K_\p;N))\\
\intertext{and define}
\HH^{d}_c(X_{K^\p\times K_\p},R)^{\epsilon}&=\HH^{d-l}(G(F),\Ah_c(K^\p\times K_\p;R(\epsilon))).
\end{align*}
As before, we often drop the subscript $R$.

Let us give a short justification for this notation:
we assume again that $K^\p\times K_\p$ is neat and $2$ is invertible in $R$.
In \cite{BS} Borel and Serre construct a contractible space $\overline{\mathcal{X}}$ containing $\mathcal{X}$ as an open subspace together with an extension of the $G(F)^+$-action.
They show that the quotient
$$\overline{\mathcal{X}}_{K^\p\times K_\p}=G(F)^+\backslash (G(\A^\infty)/(K^\p\times K_\p) \times \overline{\mathcal{X}})$$
is compact and homotopy equivalent to $\mathcal{X}_{K^\p\times K_\p}.$
Since the boundary $\partial \mathcal{X}=\overline{\mathcal{X}}\setminus \mathcal{X}$ is homotopy equivalent to the spherical building attached to $G$ by \cite{BS}, Corollary 8.4.2, a standard calculation shows that
$$\HH^{d}_c(X_{K^\p\times K_\p},N)^{\epsilon}=\HH^{d}_c(\mathcal{X}_{K^\p\times K_\p},\underline{N})^{\epsilon}.$$

Moreover, by the exactness of the complex \eqref{dualizingresolution} we get morphisms
\begin{align}\label{boundary}
\partial\colon \HH^{d}(G(F),\Ah_{R,c}(K^\p,M;N))\too \HH^{d+l}(G(F),\Ah_{R}(K^\p,M;N)),
\end{align}
which are functorial in $M$ and $N$ and compatible with the usual map from cohomology with compact support to cohomology.

\begin{Lem}\label{partial}
Let $\Omega$ be an extension of $\Q_\pi$.
\begin{enumerate}[(a)]
\item The map
$$\partial\colon \HH^{d}_c(X_{K^\p\times I_\p},\Omega)^{\epsilon}[\pi]\too \HH^{d}(X_{K^\p\times I_\p},\Omega)^{\epsilon}[\pi]$$
is an isomorphism for all characters $\epsilon$.
\item The map
$$\partial\colon \HH^{d}(G(F),\Ah_c(K^\p,v^{\infty}_I;\Omega(\epsilon)))[\pi]\too \HH^{d+l}(G(F),\Ah(K^\p,v^{\infty}_I;\Omega(\epsilon)))[\pi]$$
is an isomorphism for all subsets $I\subseteq\Delta$ and all characters $\epsilon$.
\end{enumerate}
\end{Lem}
\begin{proof}
The first claim follows from the fact that cuspidal cohomology classes lift to cohomology with compact support by the results of \cite{Bo}.
As in \eqref{evaluation} we define the map
$$\ev^{(d)}\colon \HH^{d}(\Ah_c(K^\p,\St^{\infty}_{G_\p};\Omega(\epsilon))^{\p})\too \HH^{d+l}_c(X_{K^\p\times I_\p},\Omega)^{\epsilon}$$
as evaluation at an Iwahori-fixed vector.
With the same arguments as in the proof of Proposition \ref{dimensions} we deduce that the restriction of $\ev^{(d)}$ to the $\pi$-isotypic component is an isomorphism.
The diagram
\begin{center}
\begin{tikzpicture}
    \path 	(0,0) 	node[name=A]{$\HH^{d}(G(F),\Ah_c(K^\p,St^{\infty}_{G_\p};\Omega(\epsilon))^{\p})[\pi]$}
		(6.8,0) 	node[name=B]{$\HH^{d+l}(G(F),\Ah(K^\p,St^{\infty}_{G_\p};\Omega(\epsilon))^{\p})[\pi]$}
		(0,-1.5) 	node[name=C]{$\HH^{d+l}_c(X_{K^\p\times I_\p},\Omega)^{\epsilon}[\pi]$}
		(6.8,-1.5) 	node[name=D]{$\HH^{d+l}(X_{K^\p\times I_\p},\Omega)^{\epsilon}[\pi]$};
    \draw[->] (A) -- (C) node[midway, left]{$\ev^{(d)}$};
    \draw[->] (A) -- (B) node[midway, above]{$\partial$};
    \draw[->] (C) -- (D) node[midway, above]{$\partial$};
    \draw[->] (B) -- (D) node[midway, right]{$\ev^{(d+l)}$};
\end{tikzpicture} 
\end{center}
is commutative.
Both vertical maps and the bottom horizontal one are isomorphisms and therefore, the upper horizontal map is an isomorphism as well.

As in Section \ref{Component} the short exact sequence $$0\too v^{\infty}_J(\Q)\too \mathcal{E}_{I,J}^{\infty}(\Q) \too v^{\infty}_I(\Q)\too 0$$ induces a map
$$c^{(d)}_{I,J}[\pi]_c^{\epsilon}\colon\HH^{d}(G(F),\Ah_c(K^\p,v^{\infty}_J;\Omega(\epsilon)))[\pi]
\to\HH^{d+1}(G(F),\Ah_c(K^\p,v^{\infty}_I;\Omega(\epsilon)))[\pi].$$
By the same arguments as for Corollary \ref{smoothcup} this is an isomorphism.
Again, the diagram
\begin{center}
\begin{tikzpicture}
    \path 	(0,0) 	node[name=A]{$\HH^{d}(G(F),\Ah_c(K^\p,v^{\infty}_J;\Omega(\epsilon)))[\pi]$}
		(6.8,0) 	node[name=B]{$\HH^{d+l}(G(F),\Ah(K^\p,v^{\infty}_J;\Omega(\epsilon)))[\pi]$}
		(0,-1.5) 	node[name=C]{$\HH^{d}(G(F),\Ah_c(K^\p,v^{\infty}_I;\Omega(\epsilon)))[\pi]$}
		(6.8,-1.5) 	node[name=D]{$\HH^{d+l}(G(F),\Ah(K^\p,v^{\infty}_I;\Omega(\epsilon)))[\pi]$};
    \draw[->] (A) -- (C) node[midway, left]{$c^{(d)}_{I,J}[\pi]_c^{\epsilon}$};
    \draw[->] (A) -- (B) node[midway, above]{$\partial$};
    \draw[->] (C) -- (D) node[midway, above]{$\partial$};
    \draw[->] (B) -- (D) node[midway, right]{$c^{(d+l)}_{I,J}[\pi]_c^{\epsilon}$};
\end{tikzpicture} 
\end{center}
is commutative and the vertical maps are isomorphisms.
Thus, the second claim follows by induction.
\end{proof}

Utilizing the resolution \eqref{dualizingresolution} of $D_G$ one can easily deduce that Proposition \ref{FlachundNoethersch} and Corollary \ref{limits} also hold with $\Ah_{R,c}(K^\p,M;N)$ in place of $\Ah_{R}(K^\p,M;N)$ (see Proposition 5.6 of \cite{Sp} for a detailed proof in the case $G=PGL_2$).
Let $\Omega$ be a finite extension of $\Q_\p$ containing $\Q_\pi$.
If $V\in\mathcal{U}_{\Omega}(G_\p)$ is an admissible $\Omega$-Banach representation of $G_\p$, we put
$$\Ah_{\Omega,_c}^{\cont}(K^{\p},V;\Omega(\epsilon))=\Hom_\Z(D_G,\Ah_{\Omega}^{\cont}(K^{\p},V;\Omega(\epsilon))).$$
As before we get an isomorphism
\begin{align}\label{automatic}
\HH^{d}(G(F),\Ah_c(K^\p,\St_{G_\p}^{\infty};\Omega(\epsilon)))\xlongrightarrow{\cong} \HH^{d}(G(F),\Ah_{\Omega,c}^{\cont}(K,\St_{G_\p}^{\cont};\Omega(\epsilon)))
\end{align}
of cohomology groups and hence, a well-defined cup-product pairing
\begin{align*}
&\HH^{d}(G(F),\Ah_c(K,\St_{G_\p}^{\infty};\Omega(\epsilon))) \times \Ext^{1}(v_{i}^{\infty}(\Omega),\St_{G_\p}^{\an}(\Omega))\\
\too &\HH^{d+1}(G(F),\Ah_c(K,v^{\infty}_{i};\Omega(\epsilon)))
\end{align*}
for every root $i\in \Delta$.
Thus every homomorphism $\lambda \in \Hom_{\cont}(F_\p^{\times},\Omega)$ induces a map
$$c^{(d)}_{i}(\lambda)[\pi]_c^{\epsilon}\colon \HH^{d}(G(F),\Ah_c(K,\St_{G_\p}^\infty;\Omega(\epsilon)))[\pi]\too \HH^{d+1}(G(F),\Ah_c(K,v^{\infty}_{i};\Omega(\epsilon)))[\pi].$$

\begin{Def}\label{defi}
Given a character $\epsilon\colon \pi_0(G_\infty) \to \left\{\pm 1\right\}$, an integer $d\in \Z$ with $0\leq d\leq \delta$ and a root $i\in I$ we define 
$$\LI_{i}^{(d)}(\pi,\p)_c^{\epsilon}\subseteq \Hom_{\cont}(F_\p^{\times},\Omega)$$
as the kernel of the map $\lambda \mapsto c^{(d+q-l+|J|)}_{i}(\lambda)[\pi]_c^{\epsilon}$ .
\end{Def}
 
Using similar arguments as the one in the proof of Lemma \ref{partial} one immediately proves the following claim. 
\begin{Pro}\label{compactcomparison}
Let $\epsilon\colon \pi_0(G_\infty) \to \left\{\pm 1\right\}$ be a character and $i\in \Delta$ a simple root.
We have
$$\LI_{i}^{(d)}(\pi,\p)_c^{\epsilon}=\LI_{i}^{(d)}(\pi,\p)^{\epsilon}$$
for every $0\leq d\leq \delta$.
\end{Pro}

We end this section by noting that the proof of Lemma \ref{hom} also works for cohomology with compact support.
Thus, we get:
\begin{Lem}\label{comphom}
Let $R$ be a ring and $M\in\mathfrak{C}^{\sm}_{R}(G_\p)$ a smooth $R[G_\p]$-module.
There is a natural isomorphism:
$$\HH^{0}(G(F),\Ah_{R,c}(K^\p,M;R(\epsilon)))\cong \Hom_{R[G_\p]}(M,\varinjlim_{K_\p}\HH_c^{l}(X_{K^\p\times K_\p},R)^\epsilon),$$
where the injective limit runs over all compact, open subgroups $K_\p \subseteq G_\p$.
\end{Lem}
The $R$-valued Hecke algebra $\mathbb{T}(K^\p)_R$ away from $\p$ of level $K^\p$ acts on both sides.
Going through the proof of Lemma \ref{hom} it is easy to see that the above map is Hecke-equivariant.

%%%%%%%%%%%%%%%%%%%%%
%Derived Hecke algebra
%%%%%%%%%%%%%%%%%%%%%

\section{Derived Hecke algebra}\label{DHa}
Throughout this section we fix a root $i\in\Delta$.
We show that $\LI_{i}^{(d)}(\pi,\p)^{\epsilon}$ is independent of the cohomological degree $d$ provided that a conjecture of Venkatesh on the action of the derived Hecke algebra on cohomology holds.
This generalizes the main result of \cite{Ge3}, where the case $G=PGL_2$ was treated.

We will the following notation frequently:
if $(A^d)_{d\geq 0}$ is a sequence of abelian groups, we put
$$A^\ast=\bigoplus_{d\geq 0} A^d.$$

\subsection{Local derived Hecke algebras}\label{Hecke}
We recall Venkatesh's definition of (spherical) local derived Hecke algebras and their action on various cohomology groups (cf.~Section 2 of \cite{Ve}).

Let $v$ be a prime of $F$, which does not divide $p$ and such that $G_{F_v}$ is split.
We fix a hyperspecial subgroup $K_v\subseteq G_v$.
Further, let $r\geq 1$ be an integer.
The category $\mathfrak{C}^{\sm}_{\Z/p^{r}}(G_v)$ of smooth representation of $G_v$ over $R$ has enough projective objects (see for example \cite{Ve}, Appendix A.2)
\begin{Def}
The \emph{(spherical) derived Hecke algebra} at $v$ with $\Z/p^r$-coefficients is the graded algebra
$$\Hec_{v,\Z/p^r}=\Hec(G_v,K_v)_{\Z/p^r}=\Ext^{\ast}_{\mathfrak{C}^{\sm}_{\Z/p^{r}}(G_v)}(\Z/p^r[G_v/K_v],\Z/p^r[G_v/K_v]).$$
\end{Def}

Note that the degree 0 subalgebra of $\Hec_{v,\Z/p^r}$ is the usual spherical Hecke algebra of $G_v$ with $\Z/p^r$-coefficients.

\begin{Thm}[Venkatesh]
Suppose that $p^{r}$ divides $\mathcal{N}(v)-1$ and $p$ does not divide the order of the Weil group of $G$.
Then the local derived Hecke algebra $\Hec_{v,\Z/p^r}$ is graded commutative.
\end{Thm}
\begin{proof}
This is an immediate consequence of the derived Satake isomorphism (see \cite{Ve}, Theorem 3.3)
\end{proof}

Let $\Omega$ be a finite extension of $\Q_p$ with ring of integers $R$.
Given a compact open subgroup $K^{\p}=\prod_{w\nmid \p\infty}K_w\subseteq G(\A\pinfty)$ with $K_v$ as above, a sign character $\epsilon\colon \pi_0(G_\infty)\to \left\{\pm 1\right\}$ and a smooth $G_\p$-representation $M\in \mathfrak{C}^{\sm}_{R/p^r}(G_\p) $ there is a graded action of $\Hec_{v,\Z/p^r}$ on $\HH^{\ast}(G(F),\Ah_{R/p^{r}}(K^{\p},M;R/p^{r}(\epsilon))$:
pullback along the embedding $G(F)\to G_\p$ induces a map
$$\Hec(G_v,K_v)_{\Z/p^r}\too \HH^{\ast}(G(F),\End_{\Z/p^{r}}(\Z/p^r[G_v/K_v])).$$
The action of the derived Hecke algebra is then given by the cup product pairing coming from the natural pairing
$$\Ah_{R/p^{r}}(K^{\p},M;R/p^{r}(\epsilon))\times \End_{\Z/p^{r}}(\Z/p^r[G_v/K_v]) \too \Ah_{R/p^{r}}(K^{\p},M;R/p^{r}(\epsilon)).$$
A more concrete description of the action in terms of explicit generators of the derived Hecke algebra can be found in Section 4.2 of \cite{Ge3}.

Let $k\gg 0$ be an integer.
By the results of Section \ref{Extensions4} we can associate to any locally constant homomorphism $\bar{\lambda}\colon F_\p^{\times}\to p^{k}R/p^{k+r}R$ an extension class in $\Ext^{1}_{\mathfrak{C}^{\sm}_{R/p^{r}}(G_\p)}(v^{\infty}_i(R/p^{r}),\St^{\infty}_{G_\p}(R/p^{r})).$
Taking cup product with this class induces a homomorphism

\begin{align*}
c^{(d)}_{i}(\bar{\lambda})^{\epsilon}\colon \HH^{d}(G(F),\Ah(K^\p,\St_{G_\p}^{\infty};R/p^{r}(\epsilon)))\too \HH^{d+1}(G(F),\Ah(K^\p,v^{\infty}_{i};R/p^{r}(\epsilon))).
\end{align*}

Further, let $I_\p\subseteq G_\p$ be an Iwahori subgroup.
The $\Z$-module of $I_p$-invariants of $\St_{G_\p}(\Z)$ is free of rank one (see the discussion above Corollary 2 of \cite{GK2}).
Thus, as in Section \ref{Component} a choice of a generator yields a map
\begin{align}\label{evmod}
\ev^{(d)}_{R/p^{r}}\colon \HH^{d}(G(F),\Ah(K^\p,\St^{\infty}_{G_\p};R/p^{r}(\epsilon)))\too \HH^{d}(X_{K^\p\times I_\p},R/p^{r})^{\epsilon}.
\end{align}

\begin{Lem}\label{diagrams}
Let $t\in\Hec_{v,\Z/p^r}$ be an element of degree $d^\prime$.
\begin{enumerate}[(a)]
\item\label{diagram1}
Let $\bar{\lambda}\colon F_\p^{\times}\to p^{k}R/p^{k+r}R$ be a locally constant character.
Then, the equality
$$c^{(d+d^\prime)}_{i}(\bar{\lambda})^{\epsilon}(t.x) = (-1)^{d^\prime}\cdot t.c^{(d)}_{i}(\bar{\lambda})^{\epsilon}(x)$$
holds for all $x\in \HH^{d}(G(F),\Ah(K^\p,\St_{G_\p}^{\infty};R/p^{r}(\epsilon)))$.
\item\label{diagram2}
The equality
$$\ev^{(d+d^\prime)}_{R/p^{r}}(t.x) = t.\ev^{(d)}_{R/p^{r}}(x)$$
holds for all $x\in \HH^{d}(G(F),\Ah(K^\p,\St_{G_\p}^{\infty};R/p^{r}(\epsilon)))$.
\end{enumerate}
\end{Lem}
\begin{proof}
See Lemma 4.6 of \cite{Ge3} for a detailed proof in the case $G=PGL_2$.
The same arguments carry over to this more general setup.
\end{proof}

\subsection{Global derived Hecke algebra and independence of degree}\label{Degreeshift}
As already noted in \cite{Ve}, a major technical problem in the construction of the global derived Hecke algebra is that the natural reduction maps
$$\Hec_{v,\Z/p^{n+1}}\too \Hec_{v,\Z/p^n}$$
are in general far away from being surjective.
Following Venkatesh we consider as a first step the subalgebra of the endomorphism ring of a certain cohomology group generated by the image of the action of all local derived Hecke algebras.
As a second step we take the projective limit over those subalgebras.

Let $K^{\p}\subseteq G(\A\pinfty)$ a compact open subgroup as in Section \ref{Component}.
We may assume that $K_\p$ is given as a product $\prod_{v\nmid \p\infty}K_v.$
Further, let $\Omega$ be a finite extension of $\Q_\p$, such that $\pi^{\infty}$ can be defined over $\Omega$, and $R$ its ring of integers with uniformizer $\varpi$.
We denote by $S$ the finite set of all places $v$ of $F$ that either divide $\infty$ or such that $K_v$ is not hyperspecial.
We write $\T^{(S)}=\T_R^{(S)}=R[K^S\backslash G(\A^S)/K^S]$ for the associated Hecke algebra, which is commutative by the Satake isomorphism.
It acts on $\pi^{S}$ via a character $\varphi \colon \T^{(S)} \to R$.
Let $\m=\m(\pi)$ be the kernel of the map
$$\T^{(S)}\too R/\varpi,\ t \mapstoo \varphi(t) \bmod (\varpi).$$
We put
$$A_r= \HH^{\ast}(X_{K^\p\times I_\p},R/p^{r})_{\m}^{\epsilon},$$
where the subscript $\m$ denotes localization at the maximal ideal $\m$.
Further, we set
$$A_\infty=\varprojlim_r A_r \otimes_R \Omega.$$
By Corollary \ref{limits} we have a canonical isomorphism
\begin{align*}
A_\infty[\pi]\cong \HH^{\ast}(X_{K^\p\times I_\p},\Omega)^{\epsilon}[\pi].
\end{align*}

Let $\tilde{\mathbb{T}}^{\ast}_{\Z/p^r}\subseteq \End(A_r)$ be the graded algebra generated by the actions of $\Hec_{v,\Z/p^r}$ for all places $v\notin S$ such that $G_v$ is split and $p^{r}$ divides $\mathcal{N}(v)-1$.
There are natural reduction maps $\tilde{\mathbb{T}}^{\ast}_{\Z/p^r}\to \tilde{\mathbb{T}}^{\ast}_{\Z/p^s}$ for $r\geq s$.
\begin{Def}
The \emph{global derived Hecke algebra} is the graded $\Q_p$-algebra
$$\tilde{\mathbb{T}}=\varprojlim_r \tilde{\mathbb{T}}^{\ast}_{\Z/p^r}\otimes \Q_p.$$
\end{Def}
By construction $\tilde{\mathbb{T}}$ acts on $A_\infty.$

\begin{Lem}\label{commutative}
Assume that $p$ does not divide the order of the Weil group of $G$.
Then the algebra $\tilde{\mathbb{T}}$ is graded-commutative and its action on $A_\infty$ commutes with the action of $\T^{(S)}$.
\end{Lem}
\begin{proof}
By construction it suffices to prove the statement with $\tilde{\mathbb{T}}^{\ast}_{\Z/p^r}$ in place of $\tilde{\mathbb{T}}$ respectively $A_r$ in place of $A_\infty$.
It can easily be seen that derived Hecke operators at different places (anti-)commute with each other.
Thus, the assertion follows from Venkatesh's theorem above.
\end{proof}
In particular, we see that $H^{\ast}(X_{K^\p\times I_\p},\Omega)^{\epsilon}[\pi]\cong A_\infty[\pi]$ is a $\tilde{\mathbb{T}}$-submodule of $A_\infty$.

Let us consider the module
$$B_r^{\prime}=\HH^{\ast}(G(F),\Ah(K^\p,\St_{G_\p}^{\infty};R/p^{r}R(\epsilon)))_{\m}.$$
We write $\ev^{(\ast)}_{R/p^{r}}\colon B_r^{\prime}\to A_r$ for the localization at $\m$ of the direct sums of the maps \eqref{evmod} and put
$$B_r=B_r^{\prime}/ \ker(\ev^{(\ast)}_{R/p^{r}}).$$
As above, we set
\begin{align*}
B_\infty^\prime=\varprojlim_r B_r^\prime \otimes_R \Omega
\quad \mbox{and}\quad
B_\infty=\varprojlim_r B_r \otimes_R \Omega.
\end{align*}
The evaluation maps $\ev^{(\ast)}_{R/p^{r}}$ induce an injection
$$\ev^{(\ast)}\colon B_\infty \intoo A_\infty,$$
via which we consider $B_\infty$ as a subspace of $A_\infty$.
Lemma \ref{diagrams} \eqref{diagram2} implies that $B_\infty$ is even a submodule of $A_\infty$ considered as a $\tilde{\mathbb{T}}$-module.

The $R/p^{r}$-modules $B_r^\prime$ (and thus also $B_r$) are finitely generated by Proposition \ref{FlachundNoethersch} \eqref{FuN}-Therefore, taking projective limits is exact by the Mittag-Leffler condition and one deduces that the homomorphism
$$B_\infty^{\prime}\too B_\infty$$
is surjective.
By Corollary \ref{limits} we have
\begin{align}\label{firstid}
B_\infty^\prime[\pi]\cong \HH^{\ast}(G(F),\Ah(K^\p,\St;\Omega(\epsilon)))[\pi].
\end{align}
Thus, it follows from Proposition \ref{dimensions} \eqref{seconddim} that the map
$$B_\infty^\prime [\pi]\to A_\infty[\pi]$$
induced by the evaluation maps is an isomorphism.
By definition this map factors over $B_\infty$ and, hence, we get a chain of isomorphisms
\begin{align}\label{Bisom}
B_\infty^{\prime}[\pi]\xlongrightarrow{\cong} B_\infty [\pi]\xlongrightarrow{\cong} A_\infty [\pi].
\end{align}
Note that the second isomorphisms is $\tilde{\mathbb{T}}$-equivariant.

Given a locally constant character $\bar{\lambda}\colon F_\p^{\times}\to p^{k}R/p^n$ we write
$$c^{(\ast)}_{i}(\bar{\lambda})^{\epsilon}\colon B_r^{\prime} \too \HH^{\ast}(G(F),\Ah(K^\p,v^{\infty}_{i};R/p^{r}(\epsilon)))_\m$$
for the localization at $\m$ of the direct sums of the maps $c^{d}_{i}(\bar{\lambda})^{\epsilon}$.
We define
$$C_r^{\prime}\subseteq \HH^{\ast}(G(F),\Ah(K^\p,v^{\infty}_{i};R/p^{r}(\epsilon)))$$
as the submodule generated by the images of the maps $c^{(\ast)}_{i}(\bar{\lambda})^{\epsilon}$ where $\lambda$ runs through all locally constant homomorphisms as above.
Thus, every homomorphism $\lambda$ induces a map
\begin{align}\label{modmap}
c^{(\ast)}_{i}(\bar{\lambda})^{\epsilon}\colon B_r^{\prime} \too C_r^\prime.
\end{align}
Finally, we define
$$C_r=C_r^{\prime}/\sum_{\bar{\lambda}}c^{(\ast)}_{i}(\bar{\lambda})^{\epsilon}(\ker(\ev^{(\ast)}_{R/p^{r}})).$$
The map \eqref{modmap} descends to a map
$$c^{(\ast)}_{i}(\bar{\lambda})^{\epsilon}\colon B_r\too C_r.$$

As before, we set
\begin{align*}
C_\infty^\prime=\varprojlim_r C_r^\prime \otimes_R \Omega
\quad \mbox{and}\quad
C_\infty=\varprojlim_r C_r \otimes_R \Omega.
\end{align*}
Replacing Lemma \ref{diagrams} \eqref{diagram2} by Lemma \ref{diagrams} \eqref{diagram1} one can argue as before to show that
the action of the local derived Hecke algebras on $C_r$ descend to an action of $\tilde{\mathbb{T}}$ on $C_\infty$.
Moreover, for any continuous character $\lambda\colon F_\p^{\times}\to p^{k}R$ the map
$$\varprojlim_r c^{(\ast)}_{i}(\lambda \bmod p^r)^{\epsilon} \colon B_\infty \too C_\infty$$
induced by the compatible system $\left(c^{(\ast)}_{i}(\lambda \bmod p^r)^{\epsilon}\right)_r$ is a homomorphisms of $\tilde{\mathbb{T}}$-modules.

Let us assume that $v^{\infty}_{i}(R)$ is flawless.
Then Corollary \ref{limits} and Corollary \ref{smoothcup} imply that
\begin{align}\label{secondid}
C_\infty^\prime[\pi]\cong \HH^{\ast}(G(F),\Ah(K^\p,v^{\infty}_{i};\Omega(\epsilon)))[\pi].
\end{align}
As before, since the Mittag-Leffler condition is fulfilled, we deduce that the canonical map $C_\infty^\prime\too C_\infty$
is surjective.

Moreover, for every continuous character $\lambda\colon F_\p^{\times}\to p^{k}R$ the map 
$$\varprojlim_r c^{(\ast)}_{i}(\lambda \bmod p^r)^{\epsilon} [\pi]\colon B_\infty^\prime[\pi] \too C_\infty^\prime[\pi]$$
induced by the compatible system $\left(c^{(\ast)}_{i}(\lambda \bmod p^r)^{\epsilon}\right)_r$ agrees with $c^{(\ast)}_{i}(\lambda )^{\epsilon}$ under the identifications \eqref{firstid} and \eqref{secondid}.
Hence, by Corollary \ref{smoothcup} the map above is an isomorphism for $\lambda$ a suitable multiple of the $p$-adic valuation.
One deduces from the first isomorphism in \eqref{Bisom} that the map
$$C_\infty^\prime[\pi]\too C_\infty[\pi]$$
is an isomorphism and, therefore, we have an isomorphism
$$C_\infty[\pi]\cong \HH^{\ast}(G(F),\Ah(K^\p,v^{\infty}_{i};\Omega(\epsilon)))[\pi]$$
by \eqref{secondid}.

\begin{Thm}
Assume that $p$ does not divide the order of the Weil group of $G$ and that $v^{\infty}_{i}(R)$ is flawless.
Further, assume that $\LI_{i}^{(0)}(\pi,\p)^{\epsilon}\subseteq \Hom_{\cont}(F_\p^{\times},\Omega)$ is a subspace of codimension one.
If $\HH^{q}(X_{K^\p\times I_\p},\Omega)^{\epsilon}[\pi]$ generates $\HH^{\ast}(X_{K^\p\times I_\p},\Omega)^{\epsilon}[\pi]$ as a $\tilde{\mathbb{T}}$-module, then part \eqref{conjdegree} of Conjecture \ref{conjecture} holds, i.e.~we have
$$\LI_{i}^{(d)}(\pi,\p)^{\epsilon}=\LI_{i}^{(0)}(\pi,\p)^{\epsilon}$$
for all $0\leq d \leq \delta$.
\end{Thm}
\begin{proof}
As explained above, we have endowed the source and the target of the isomorphism
$$\HH^{\ast}(G(F),\Ah(K^\p,\St^{\infty}_{G_\p};\Omega(\epsilon)))[\pi]\xrightarrow{\ev^{(d)}} \HH^{\ast}(X_{K^\p\times I_\p},\Omega)^{\epsilon}[\pi]$$
with an action of the global derived Hecke algebra $\tilde{\mathbb{T}}$ in such a way that the map is $\tilde{\mathbb{T}}$-equivariant.
As the right hand side is generated by its lowest degree component, so is the left hand side.
Let $m$ be an element of $\HH^{d}(G(F),\Ah(K^\p,\St^{\infty}_{G_\p};\Omega(\epsilon)))[\pi].$
By the above we find an element $x\in\HH^{q}(G(F),\Ah(K^\p,\St^{\infty}_{G_\p};\Omega(\epsilon)))[\pi]$ and a derived Hecke operator $t\in \tilde{\mathbb{T}}$ of degree $d-q$ such that $t.x=m.$

Let $\lambda\colon F_\p^\times\to \Omega$ be a continuous character.
By our codimension one assumption we know that there exists a constant $\ell_{\lambda}\in\Omega$ such that
$$c^{(q)}_{i}(\lambda)[\pi]^{\epsilon}(y)=\ell_\lambda\cdot  c^{(q)}_{i}(\ord_\p)[\pi]^{\epsilon}(y)$$
for all $y\in \HH^{q}(G(F),\Ah(K^\p,\St^{\infty}_{G_\p};\Omega(\epsilon)))[\pi].$
Thus, a homomorphism $\lambda$ is an element of  $\LI_{i}^{(0)}(\pi,\p)^{\epsilon}$ if and only if $\ell_\lambda=0.$
From the discussion preceding the theorem and Lemma \ref{diagrams} \eqref{diagram1} we get that
\begin{align*}
c^{(d)}_{i}(\lambda)[\pi]^{\epsilon}(m)
&= c^{(d)}_{i}(\lambda)[\pi]^{\epsilon}(t.x)\\
&= (-1)^{d-q}\cdot t.c^{(q)}_{i}(\lambda)[\pi]^{\epsilon}(x)\\
&= (-1)^{d-q}\ell_\lambda\cdot  t.c^{(q)}_{i}(\ord_\p)[\pi]^{\epsilon}(x)\\
&= \ell_\lambda\cdot c^{(d)}_{i}(\ord_\p)[\pi]^{\epsilon}(t.x)\\
&= \ell_\lambda\cdot c^{(d)}_{i}(\ord_\p)[\pi]^{\epsilon}(m)
\end{align*}
which proves the assertion.
\end{proof}

\begin{Rem}
By Proposition \ref{codimension} the condition that $\LI_{i}^{(0)}(\pi,\p)^{\epsilon}$ has codimension one is fulfilled if the multiplicity of the representation $\pi$ is equal to one, which always holds if $G=PGL_n$.
By Theorem \ref{admissible} the $R[G_\p]$-module $v^{\infty}_{i}(R)$ is also flawless in this case.
Thus, if $G=PGL_n$, the assumptions of the above theorem reduce to $p\nmid n!$ and that the $\pi$-isotypic component of cohomology of the associated locally symmetric space is generated by its minimal degree component.
The later assumption is, of course, the crucial one.
In case that $\pi$ is spherical at $p$ it was formulated as a question by Venkatesh in \cite{Ve}.
See Theorem 7.6 of \textit{loc.cit.~}for certain cases, in which Venkatesh answers this question affirmatively. 
\end{Rem}

%%%%%%%%%%%%%%%%%%
%Galois
%%%%%%%%%%%%%%%%%%

\section{Galois representations and completed cohomology}\label{someconjectures}
\subsection{Conjectural connection to Galois representations}\label{Galois}
By assumption the automorphic representation $\pi$ is cohomological.
So, in particular, it is $C$-algebraic in the sense of Buzzard and Gee (cf.~\cite{BuzzGee}).
In \textit{loc.~cit.~}it is conjectured that one can associate a Galois representation $\rho_\pi$ to $\pi$ that takes values in $p$-adic points of the $C$-group $\CG$ of $G$.
After recalling the definition of the $C$-group we will refine the conjecture in our case.
More precisely, we will give a conjectural description of certain subquotients of the restriction $\rho_{\pi,\p}$ of $\rho_\pi$ to the local Galois group at $\p$ in terms of automorphic $\LI$-invariants.

Let us start with a brief reminder on the classical situation.
Suppose $f$ is a (elliptic modular) newform of weight $2$ and level $\Gamma_0(N)$, which is Steinberg at the rational prime $p$.
For simplicity, we assume that $f$ has rational Fourier coefficients
The automorphic representation $\pi_f$ associated to $f$ is a representation of $PGL_2$.
Although the dual group of $PGL_2$ is $SL_2$, the $p$-adic Galois representation $\rho_f$ attached to $f$ takes values in $GL_2(\Q_p)$ and no twist of it has image in $SL_2(\Q_p)$.
We write $\Q_p(1)$ for the one-dimensional Galois representation given by the cyclotomic character (and use a similar notation in case the coefficients are a finite extension of $\Q_p$).
Since $f$ is Steinberg (and ordinary) at $p$ the restriction of $\rho_f$ to the local Galois group $\Gal(\overline{\Q}_p/\Q_p)$ at $p$ is a non-split extension of $\Q_p$ by $\Q_p(1)$, i.e.~it defines a non-zero class $[\rho_{f,p}]$ in $\HH^{1}(\Gal(\overline{\Q}_p/\Q_p),\Q_p(1)).$
By local class field theory we have an isomorphism
$$\Hom_{\cont}(\Q_p^{\times},\Q_p)\cong\HH^{1}(\Gal(\overline{\Q}_p/\Q_p),\Q_p).$$
We define
$$\LI(\rho_{f,p})\subseteq \Hom_{\cont}(\Q_p^{\times},\Q_p)$$
as the kernel of taking the cup product with the class $[\rho_{f,p}]$.
By local Tate duality this is subspace of codimension one.
By Theorem 1 of \cite{D} we have an equality of subspaces
\begin{align}\label{Darmon}
\LI(\rho_{f,p})=\LI(\pi_f,p).
\end{align}
In this situation there is only one simple root and one cohomological degree.
Hence, we have dropped them from the notation.
Since the independence of the sign character is known for the automorphic $\LI$-invariant in this case (see Remark \ref{independenceRem}), we dropped it from the notation as well.
Over totally real number fields only partial results on the equality of automorphic and arithmetic $\LI$-invariants for modular elliptic curves are known (see Proposition 5.9 of \cite{Sp}).

Now let us return to our general setup.
In Proposition 5.3.1 of \cite{BuzzGee} a canonical $\mathbb{G}_m$-extension $G^{\prime}$ of $G$ is constructed.
For example, in the case $G=PGL_2$ this extensions turns out to be $GL_2$.
The $C$-group $\CG$ of $G$ is defined to be the $L$-group of $G^{\prime}$.
Here we view the $L$-group of a reductive group over $F$ as an algebraic group defined over $\overline{\Q}$.
Thus, the group $\CG$ is a semi-direct product of the dual reductive group $\widehat{G^{\prime}}$ of $G^{\prime}$ (defined over $\overline{\Q}$) and the absolute Galois group $\Gal(\overline{\Q}/F)$ of $F$.
Let
$$\rho_\pi\colon\Gal(\overline{\Q}/F)\too \CG(\overline{\Q}_p)$$
be one of the Galois representation conjecturally associated to $\pi$ as in \cite{BuzzGee}, Conjecture 5.3.4.
In particular, the composition of $\rho_\pi$ with the natural projection $\CG(\overline{\Q}_p)\to \Gal(\overline{\Q}/F)$ is the identity and the restriction of $\rho_\pi$ to almost all local Galois groups is compatible with the local Satake isomorphism (cf.~\textit{loc.~cit.~}for more details.)
Note that $\rho_\pi$ is not uniquely determined (up to conjugation) by $\pi$ as pointed out in Remark 3.2.4 of \cite{BuzzGee}.

By assumption, $G$ and hence $G^{\prime}$ is split at $\p$.
Therefore, the action of the local Galois group $\Gal(\overline{\Q}_p/F_\p)$ on $\widehat{G^{\prime}}$ is trivial.
Thus, restricting $\rho_\pi$ to the local Galois group at $\p$ yields a homomorphism
$$\rho_{\pi,\p}\colon\Gal(\overline{\Q}_p/F_\p)\too \widehat{G^{\prime}}(\overline{\Q}_p) \times \Gal(\overline{\Q}_p/F_\p)\too \widehat{G^{\prime}}(\overline{\Q}_p),$$
where the second map is given by projection on the first factor.
We may fix a finite extension $\Omega$ of $\Q_p$ such that $\rho_{\pi,\p}$ takes values in $\widehat{G^{\prime}}(\Omega)$.

We fix a split torus and a Borel subgroup $\widehat{T^{\prime}}\subseteq\widehat{B^{\prime}}\subseteq\widehat{G^{\prime}}$ and view $\Delta$ as a set of cocharacters of $T$.
Let $\Delta^{\vee}$ be the associated set of dual roots. 

\begin{Def}
A representation
$$\rho_{\p}\colon\Gal(\overline{\Q}_p/F_\p)\too \widehat{G^{\prime}}(\Omega)$$
is \emph{special} if 
\begin{itemize}
\item its image is (up to conjugation) contained in $\widehat{B^{\prime}}(\Omega)$ and
\item the composition $i^{\vee}\circ\rho_{\p}$ is the cyclotomic character for all $i^{\vee}\in \Delta^{\vee}.$
\end{itemize}
\end{Def}

Let $\rho_{\p}\colon\Gal(\overline{\Q}_p/F_\p)\to \widehat{G^{\prime}}(\Omega)$ be a special representation.
For every $i\in \Delta$ let $\widehat{P^\prime_i}\subseteq \widehat{G^{\prime}}$ the standard parabolic subgroup associated to its dual root $i^{\vee}$.
Since the Levi subgroup of $\widehat{P^\prime_i}$ has semi-simple rank 1, there exists a non-zero map
$$\pr_i\colon \widehat{P^\prime_i}\too PGL_2.$$
The composition $\rho_{\p,i}=\pr_i\circ\rho_{\p}$ takes (up to conjugation) values in the Borel subgroup of upper triangular matrices in $PGL_2(\Omega)$.
Since $\rho_{\p}$ is special we can lift $\rho_{\p,i}$ to an honest two-dimensional representation which is an extension of $\Omega$ by $\Omega(1)$
and, thus, it defines a class $[\rho_{\p,i}]$ in $\HH^{1}(\Gal(\overline{\Q}_p/\Q_p),\Omega(1)).$
Note that a priori this class could be trivial.
Similar as above, we define
$$\LI_{i}(\rho_{\p})\subseteq \Hom_{\cont}(F_\p^{\times},\Omega)\cong\HH^{1}(\Gal(\overline{\Q}_p/F_\p),\Omega)$$
as the kernel of taking the cup product with the class $[\rho_{\p,i}]$.
By local Tate duality this is a subspace of codimension at most one and the representation $\rho_{\p,i}$ is crystalline if and only if $\ord_\p \in \LI_{i}(\rho_{\p}).$

In accordance with \eqref{Darmon} we make the following prediction.
\begin{Con}\label{Galoisconj}
The representation $\rho_{\pi,\p}$ is special and the equality
$$\LI_{i}(\rho_{\pi,\p})=\LI_{i}^{(0)}(\pi,\p)^{\cf}$$
holds for all $i\in\Delta$.
\end{Con}
If Conjecture \ref{Galoisconj} holds, Proposition \ref{codimension} would imply that $\LI_{i}(\rho_{\pi,\p})$ does not contain the homomorphism $\ord_\p$ and, therefore, the two-dimensional representation $\rho_{\pi,\p,i}$ is not crystalline for all $i\in\Delta$.

Let us give some evidence for the first part of the conjecture, which follows from deep results in the Langlands program.
We assume that $F$ is either totally real or a CM field and that $G$ is equal to $PGL_n$.
In addition to our running assumptions we assume that the automorphic representation $\pi$ is polarized.
Then, by the work of many people (see Theorem 2.1.1 of \cite{BLGGT} for a detailed discussion) there exists a $p$-adic Galois representation
$$\rho_{\pi}\colon \Gal(\overline{\Q}/F)\too GL_n(\Omega),$$
which is compatible with the local Langlands correspondence at all finite places away from $p$ and is Hodge--Tate with weights $n-1, n-2,\ldots,0$ at all places above $p$.
(Note that $GL_n$ is not isomorphic but merely isogenous to $\widehat{PGL_n^\prime}$.)
Since $\pi_\p=\St_\p^{\infty}(\C)$ has an Iwahori-fixed vector, the local representation $\rho_{\pi,\p}$ is semi-stable and the associated Weil--Deligne representation is equal to the image of $\St_\p^{\infty}(\C)$ under the local Langlands correspondence by the main theorem of \cite{Caraiani}.
In particular, the associated Newton and Hodge polygons agree.
Thus, one can deduce that $\rho_{\pi,\p}$ is special. 
In addition, one sees that the monodromy operator has maximal rank and, therefore, an easy calculation shows that the representations $\rho_{\pi,\p,i}$, $i=1,\ldots,n-1$, are not crystalline (cf.~\cite{Ding}, Lemma 3.2). 

In Section \ref{Galoisproof} we will provide a proof of the above conjecture for certain automorphic representations of definite unitary groups. 

\begin{Rem}
In case of a non-trivial coefficient system $\pi$ would no longer be ordinary at $\p$.
Thus, one does not expect that the associated local Galois representation is upper triangular but rather trianguline.
For this to make sense and to define $\LI$-invariants on the Galois side for a general reductive group $G$ one needs to develop a theory of $\widehat{G^{\prime}}$-valued $(\varphi,\Gamma)$-modules.
\end{Rem}

\subsection{Completed cohomology - relatively definite case}\label{Completed}
We relate automorphic $\LI$-invariants to the $\pi$-isotypic component of completed cohomology as introduced by Emerton in \cite{Em}.
First, we deal with the case that the group $G$ is close to being definite.
The general case will be discussed in a subsequent section.

As before, let $\Omega$ be a finite extension of $\Q_p$, which is large enough, and $R$ its ring of integers with uniformizer $\varpi$.
Also, we fix a compact, open subgroup $K^\p\subseteq G(\A\pinfty)$, which is small enough. 

We put
$$\widetilde{\HH}{}_c^{d}(X_{K^\p},R/\varpi^r)^\epsilon=\varinjlim_{K_\p}\HH_c^{d}(X_{K^\p\times K_\p},R/\varpi^r)^\epsilon,$$
where the injective limit runs over all compact, open subgroups $K_\p \subseteq G_\p$.
Completed cohomology (with compact support) of tame level $K^{p}$, degree $d$ and sign $\epsilon$ with coefficients in $R$ respectively $\Omega$ is defined by
\begin{align*}
\widetilde{\HH}{}_c^{d}(X_{K^\p},R)^\epsilon&=\varprojlim_{r}\widetilde{\HH}{}_c^{d}(X_{K^\p},R/\varpi^r)^\epsilon\\
\intertext{respectively}
\widetilde{\HH}{}_c^{d}(X_{K^\p},\Omega)^\epsilon&=\widetilde{\HH}{}_c^{d}(X_{K^\p},R)^\epsilon \otimes \Omega.
\end{align*}
The completed cohomology group $\widetilde{\HH}{}_c^{d}(X_{K^\p},\Omega)^\epsilon$ is an admissible unitary Banach representation of $G_\p$ by \cite{Em}, Theorem 2.1.1.

The following theorem was first proven in the case $G=GL_2$, $F=\Q$ by Breuil (see \cite{Br2}, Theorem 1.1.3).
\begin{Thm}\label{compcomp}
Let $B$ be a $\Omega$-Banach representation of $G_\p$.
There is a Hecke-equivariant isomorphism
$$\HH^{0}(G(F),\Ah_{\Omega,c}^{\cont}(K^\p,B;\Omega(\epsilon)))\cong \Hom_{\Omega[G_\p],\cont}(B,\widetilde{\HH}{}_c^{l}(X_{K^\p},\Omega)^\epsilon),$$
which is functorial in $B$.
\end{Thm}
\begin{proof}
Let us fix a $G_\p$-invariant complete $R$-lattice $M\subseteq B$ and put $M_r=M/\varpi^{r}M.$
Since $M_r$ is a smooth $G_\p$-module, we have
$$\HH^{0}(G(F),\Ah_{R,c}(K^\p,M_r;R/\varpi^{r}(\epsilon)))\cong\Hom_{R[G_\p]}(M_r,\widetilde{\HH}{}_c^{d}(X_{K^\p},R/\varpi^r)^\epsilon)$$
by Lemma \ref{comphom}.
Note that we may replace $M_r$ by $M$ on both sides without changing them.
Taking projective limits we get an isomorphism
$$\HH^{0}(G(F),\Ah_{R,c}(K^\p,M;R(\epsilon)))\cong\Hom_{R[G_\p]}(M,\widetilde{\HH}{}_c^{d}(X_{K^\p},R)^\epsilon).$$
It is known that the torsion submodule of completed cohomology has bounded exponent (cf.~\cite{CaEm}, Theorem 1.1).
Thus, we may localize at $p$ to obtain the sought-after isomorphism.
It obviously commutes with the action of the Hecke algebra.
\end{proof}

\begin{Def}
We call the group $G$ \emph{relatively definite} if the equality $l=q$ holds.
\end{Def}

\begin{Exa}
If $G$ is definite, then both $l$ and $q$ are zero and, therefore, $G$ is relatively definite.
Suppose $F=\Q$, then the groups $PGL_2$ and $PGL_3$ are relatively definite.
In the first case we have $l=q=1$ and in the second case we have $l=q=2$.
If $F$ is an imaginary quadratic field, the group $PGL_2$ is relatively definite with $l=q=1$.
Every group isogenous to a relatively definite group is relatively definite.
\end{Exa}

Combining Theorem \ref{compcomp} with Lemma \ref{partial}, Proposition \ref{dimensions} \eqref{thirddim} and the isomorphism \eqref{automatic}
we get the following result.
\begin{Cor}\label{fifthdimension}
Suppose $G$ is relatively definite.
Then we have:
$$\dim \Hom_{\Omega[G_\p],\cont}(\St_{G_\p}^{\cont},\widetilde{\HH}{}_c^{q}(X_{K^\p},\Omega)^\epsilon[\pi])= m_\pi.$$
\end{Cor}

\begin{Pro}\label{last}
Suppose $G$ is relatively definite.
Under the assumption that $v_i^{\infty}(R)$ is flawless, we have:
The canonical restriction map
$$\Hom_{\Omega[G_\p],\cont}(\mathcal{E}^{\cont}_{i}(\lambda),\widetilde{\HH}{}_c^{q}(X_{K^\p},\Omega)^\epsilon[\pi])\too \Hom_{\Omega[G_\p],\cont}(\St_{G_\p}^{\cont}(\Omega),\widetilde{\HH}{}_c^{q}(X_{K^\p},\Omega)^\epsilon[\pi])$$
is injective for all homomorphisms $\lambda\in \Hom_{\cont}(F_\p^{\times},\Omega)$.
It is an isomorphism if and only if $\lambda\in \LI_{i}^{(0)}(\pi,\p)^{\epsilon}$.
%In particular, if $m_\pi=1$, we have
%$$\LI_{i}^{\Br}(\pi,\p)^{\epsilon}=\LI_{i}^{(0)}(\pi,\p)^{\epsilon}.$$
\end{Pro}
\begin{proof}
Taking continuous duals is an exact functor on Banach spaces.
Therefore, we have a long exact sequence
\begin{align*}\ldots\too\HH^{d}(G(F),\Ah^{\cont}_{c}(K^\p,\VV_i^{\cont};\Omega(\epsilon))[\pi]
\too\HH^{d}(G(F),\Ah^{\cont}_{c}(K^\p,\mathcal{E}^{\cont}_{i}(\lambda);\Omega(\epsilon))[\pi]\\
\too\HH^{d}(G(F),\Ah^{\cont}_{c}(K^\p,\St^{\cont}_{G_\p};\Omega(\epsilon))[\pi]
\too\HH^{d+1}(G(F),\Ah^{\cont}_{c}(K^\p,\VV_i^{\cont};\Omega(\epsilon))[\pi]
\too\ldots\end{align*}
for every $\lambda\in \Hom_{\cont}(F_\p^{\times},\Omega)$.
We have
$$\HH^{0}(G(F),\Ah^{\cont}_{c}(K^\p,\VV_i^{\cont};\Omega(\epsilon))[\pi]=\HH^{0}(G(F),\Ah^{\cont}_{c}(K^\p,v_i^{\infty};\Omega(\epsilon))[\pi]=0,$$
where the first equality follows by flawlessness of $v_i^{\infty}(R)$ and the second by Proposition \ref{fourthdim} and Lemma \ref{partial}.

Since $v_i^{\infty}(R)$ is flawless, we may take the cup product with $\mathcal{E}^{\cont}_{i}(\lambda)$ to define automorphic $\LI$-invariants.
Hence, the first boundary map in the above exact sequence is zero if and only if $\lambda\in \LI_{i}^{(0)}(\pi,\p)_c^{\epsilon}$.
But by Proposition \ref{compactcomparison} we have an equality of $\LI$-invariants $\LI_{i}^{(0)}(\pi,\p)^{\epsilon}=\LI_{i}^{(0)}(\pi,\p)_c^{\epsilon}$.
Now the claims follow by applying Theorem \ref{compcomp}.
\end{proof}
\begin{Rem}
\begin{enumerate}[(i)]\thmenumhspace
\item In the case of elliptic modular cusp forms the proposition above is due to Breuil (see \cite{Br2}, Theorem 1.1.5).
\item One might wonder if a similar statement is true in higher cohomological degree.
But by the vanishing conjectures on completed cohomology dictate that this cannot be the case.
More precisely, it is conjectured that $\widetilde{\HH}{}_c^{d}(X_{K^\p},\Omega)^\epsilon=0$ for all $d> q$ and every group $G$ (see the survey paper of Calegari and Emerton \cite{CaEm} for a more detailed conjecture).
This is known for groups of Hodge type by Scholze (see \cite{Scholze}, Corollary IV.2.2).
A non-Shimura example for which the vanishing conjecture can be checked easily is the group $PGL_2$ over an imaginary quadratic field (cf.~Example 2.12 of \cite{CaEm}).
\end{enumerate}
\end{Rem}

\subsection{Definite unitary case}\label{Galoisproof}
Let us assume that we are in the situation considered in Section 4 of \cite{Ding} (see also Section 5.1 of \cite{EmGee} and Section 2.3 of \cite{six}). 
We do not spell out the complete list of assumptions imposed on $G$ and $\pi$ in these references but give a summary of the most crucial ones in the following. 
We assume $F$ is a totally real number field and we fix a totally imaginary quadratic extension $E$ of $F$, which is unramified at all finite places and split at $\p$.
Let $\q$ be any prime of $E$ lying above $\p$. 
Aside from our usual assumption on $G$ we suppose that $G$ is a definite outer form of $PGL_n$ such that $G_E=PGL_{n,E}$ and $p \nmid 2n$.
(In particular, $\delta=0$ and $\pi_0(G_\infty)$ is trivial and, therefore, there is only a single automorphic $\LI$-invariant attached to $\pi$.)

We further assume that $\pi_v$ is spherical at all finite places $v$ except at $v=\p$ and at another finite place $v_1\nmid p$, which is split in $E$ and does not split completely in $F(\zeta_p)$. 

Let $\rho_{\pi_E}\colon \Gal(\overline{\Q}/E)\to GL_n(\Omega)$ be the $p$-adic Galois representation associated (to the base change of) $\pi$ (see \cite{ChH} and\cite{Shin} for its existence).
Aside some other assumptions on the reduction of $\rho_{\pi,E}$ we assume that it is unramified and adequate, hence in particular irreducible, and that Frobenius at $v_1$ acts with distinct eigenvalues, which are not equal to $\mathcal{N}(v_1)^{\pm 1}.$

\begin{Thm}\label{Galequality}
Under the above assumptions Conjecture \ref{Galoisconj} holds, i.e.~we have the equality
$$\LI_{i}(\rho_{\pi_E,\q})=\LI_{i}^{(0)}(\pi,\p)^{\cf}$$
for all $i\in\Delta$.
\end{Thm}
\begin{proof}
Since $G$ is definite Proposition \ref{last} applies, i.e.~a continuous homomorphism $\lambda\colon F_\p^{\times}\to \Omega$ lies in $\LI_{i}^{(0)}(\pi,\p)^{\cf}$ if and only if every embedding $\St^{\cont}_{G_\p}\into \widetilde{\HH}{}_c^{0}(X_{K^\p},\Omega)^\epsilon[\pi]$ can be extended to a map $\mathcal{E}^{\cont}_{i}(\lambda)\to \widetilde{\HH}{}_c^{0}(X_{K^\p},\Omega)^\epsilon[\pi].$ 
(Note that by Theorem \ref{genflaw} the flawlessness assumption of $v_i^\infty(R)$ is fulfilled.)
But by Theorem 1.2 of \cite{Ding} this is equivalent to $\lambda\in \LI_{i}(\rho_{\pi_E,\q})$.
\end{proof}

\subsection{Completed cohomology - general case}\label{Spectral}
We end by making a few remarks on how the results from Section \ref{Completed} can be generalized.
A priori, it might be possible that the $\pi$-isotypic part of completed cohomology is non-zero below the expected degree, i.e.~ below $q$.
In order to avoid this problem, we make the following strong assumption, which is, in general, hard to check as the example at the end of the section demonstrates.
Let $\m$ be the maximal ideal of the integral Hecke algebra $\T^{(S)}=\T_R^{(S)}$ as defined at the beginning of Section \ref{Degreeshift}.
\begin{Ass}\label{assvanish}
We assume that
$$\HH_c^{d}(X_{K^\p\times K_\p},\Z/p)_\m^\epsilon=0$$
for all $0\leq d<q$ and all open, compact subgroups $K_\p\subseteq G_\p$.
\end{Ass}
Under this assumption one can inductively prove that $\widetilde{\HH}{}_c^{d}(X_{K^\p},R/\varpi^{r})_\m^\epsilon=0$ holds for all $r\geq 1$.
Then, with an easy spectral sequence argument one can the following generalization of Lemma \ref{comphom}.
\begin{Lem}
Suppose that Assumption \ref{assvanish} holds.
There is a natural isomorphism
$$\HH^{q-l}(G(F),\Ah_{R,c}(K^\p,M;R/\varpi^{r}(\epsilon)))_\m\cong \Hom_{R[G_\p]}(M,\widetilde{\HH}{}_c^{q}(X_{K^\p},R/\varpi^r)_\m^\epsilon)$$
for every $M\in\mathfrak{C}^{\sm}_{R/\varpi^{r}}(G_\p)$.
\end{Lem}
Using the lemma above one gets the following generalizations of the results of Section \ref{Completed}.
\begin{Thm}
Suppose that Assumption \ref{assvanish} holds.
Let $B$ be a $\Omega$-Banach representation of $G_\p$.
There is a Hecke-equivariant isomorphism
$$\HH^{q-l}(G(F),\Ah_{\Omega,c}^{\cont}(K^\p,B;\Omega(\epsilon)))_\m\cong \Hom_{\Omega[G_\p]}(B,\widetilde{\HH}{}_c^{q}(X_{K^\p},\Omega)_\m^\epsilon).$$
\end{Thm}

\begin{Cor}
Suppose that Assumption \ref{assvanish} holds.
We have:
$$\dim_\Omega \Hom_{\Omega[G_\p],\cont}(\St_{G_\p}^{\cont},\widetilde{\HH}{}_c^{q}(X_{K^\p},\Omega)^\epsilon[\pi])= m_\pi.$$
\end{Cor}
\begin{Pro}\label{lastlast}
Suppose that Assumption \ref{assvanish} holds and that $v_i(R)$ is flawless.
The canonical restriction map
$$\Hom_{\Omega[G_\p],\cont}(\mathcal{E}^{\cont}_{i}(\lambda),\widetilde{\HH}{}_c^{q}(X_{K^\p},\Omega)^\epsilon[\pi])\too \Hom_{\Omega[G_\p],\cont}(\St_{G_\p}^{\an}(\Omega),\widetilde{\HH}{}_c^{q}(X_{K^\p},\Omega)^\epsilon[\pi])$$
is injective for all homomorphisms $\lambda\in \Hom_{\cont}(F_\p^{\times},\Omega)$.
It is an isomorphism if and only if $\lambda\in \LI_{i}^{(0)}(\pi,\p)^{\epsilon}$.
\end{Pro}

\begin{Rem}
Let us end with mentioning an example, where the above assumption on the vanishing of completed cohomology is satisfied.
Let us assume that $F$ is totally real and $G$ is an anisotropic similitude group of a quadratic CM extension $E$ of $F$, which contains an imaginary quadratic field.
Further, we assume that $G$ is associated with a division algebra over $E$.
The main theorem of \cite{CaSch} states that Assumption \ref{assvanish} holds if there exists a rational prime $l\neq p$ that is completely split in $F$ such that the $\bmod\ p$ Galois representation attached to $\pi$ is unramified and decomposed generic at all places above $l$.
Being decomposed generic is a certain big image condition (see Definition 1.9 of \textit{loc.cit.~}for a precise definition).
\end{Rem}

\def\cprime{$'$}

\end{document}